\documentclass[11pt]{amsart}

\usepackage{amssymb}
\usepackage{url}
\usepackage{adjustbox}

\usepackage{wasysym}
\usepackage{amsmath, calc, mathdots, physics}

\usepackage{pgfplots}
\pgfplotsset{width=10cm,compat=1.9}

\usepackage{mathrsfs,multicol}
\usepackage[pagebackref, 
hypertexnames=false, colorlinks, citecolor=red, linkcolor=red]{hyperref}
\usetikzlibrary{decorations.pathmorphing}

	\normalbaselines						  %
\newcommand{\wt}{\widetilde}

\def\TT{\mathbb{T}}
\def\RR{\mathbb{R}}
\def\ZZ{\mathbb{Z}}
\def\CC{\mathbb{C}}
\def\NN{\mathbb{N}}
\def\SS{\mathbb{S}}

\newcommand{\la}{{\lambda}}
\newcommand{\f}{{\varphi}}

\newcommand{\cU}{{\mathcal{U}}}

\newcommand{\R}{{\mathbb  R}}

\newcommand{\te}{{\theta}}

\makeatletter
\def\Ddots{\mathinner{\mkern1mu\raise\p@
\vbox{\kern7\p@\hbox{.}}\mkern2mu
\raise4\p@\hbox{.}\mkern2mu\raise7\p@\hbox{.}\mkern1mu}}
\makeatother

\newcommand{\cF}{\mathcal{F}}
\newcommand{\cA}{\mathcal{A}}
\newcommand{\eps}{\varepsilon}

\newcommand{\bd}{\boldsymbol{d}}

\newcommand{\cD}{\mathcal{D}}

\newcommand{\cE}{\mathcal{E}}
\newcommand{\cR}{\mathcal{R}}

\newcommand{\cM}{\mathcal{M}}

\newcommand{\fD}{\mathfrak{D}}

\newcommand{\fM}{\mathfrak{M}}

\newcommand{\cI}{\mathcal{I}}

\DeclareMathOperator{\Ker}{Ker}

\DeclareMathOperator{\dom}{dom}

\DeclareMathOperator{\Ran}{ran}

\newcommand{\cS}{\mathcal{S}}
\newcommand{\II}{\mathbb{I}}
\newcommand{\w}{\omega}

\newcommand{\ci}[1]{_{ {}_{\scriptstyle #1}}}
\newcommand{\ti}[1]{_{\scriptstyle \text{\rm #1}}}

\definecolor{darkblue}{rgb}{0.2,0.2,0.6}
\definecolor{darkgreen}{rgb}{0.2,0.6,0.2}

\catcode`@=11
\chardef\mathlig@atcode\count255

\def\actively#1#2{\begingroup\uccode`\~=`#2\relax\uppercase{\endgroup#1~}}
\def\mathlig@gobble{\afterassignment\mathlig@next@cmd\let\mathlig@next= }
\def\mathlig@delim{\mathlig@delim}
\def\mathlig@defcs#1{\expandafter\def\csname#1\endcsname}
\def\mathlig@let@cs#1#2{\expandafter\let\expandafter#1\csname#2\endcsname}
\def\mathlig@appendcs#1#2{\expandafter\edef\csname#1\endcsname{\csname#1\endcsname#2}}

\def\mathlig#1#2{\mathlig@checklig#1\mathlig@end\mathlig@defcs{mathlig@back@#1}{#2}\ignorespaces}

\def\mathlig@checklig#1#2\mathlig@end{%
 \expandafter\ifx\csname mathlig@forw@#1\endcsname\relax
 \expandafter\mathchardef\csname mathlig@back@#1\endcsname=\mathcode`#1%
 \mathcode`#1"8000\actively\def#1{\csname mathlig@look@#1\endcsname}%
 \mathlig@dolig#1\mathlig@delim
\fi
\mathlig@checksuffix#1#2\mathlig@end
}

\def\mathlig@checksuffix#1#2\mathlig@end{%
\ifx\mathlig@delim#2\mathlig@delim\relax\else\mathlig@checksuffix@{#1}#2\mathlig@end\fi
}
\def\mathlig@checksuffix@#1#2#3\mathlig@end{%
\expandafter\ifx\csname mathlig@forw@#1#2\endcsname\relax\mathlig@dosuffix{#1}{#2}\fi
\mathlig@checksuffix{#1#2}#3\mathlig@end
}

\def\mathlig@dosuffix#1#2{%
\mathlig@appendcs{mathlig@toks@#1}{#2}%
\mathlig@dolig{#1}{#2}\mathlig@delim
}

\def\mathlig@dolig#1#2\mathlig@delim{%
 \mathlig@defcs{mathlig@look@#1#2}{%
 \mathlig@let@cs\mathlig@next{mathlig@forw@#1#2}\futurelet\mathlig@next@tok\mathlig@next}%
 \mathlig@defcs{mathlig@forw@#1#2}{%
  \mathlig@let@cs\mathlig@next{mathlig@back@#1#2}%
  \mathlig@let@cs\checker{mathlig@chck@#1#2}%
  \mathlig@let@cs\mathligtoks{mathlig@toks@#1#2}%
  \expandafter\ifx\expandafter\mathlig@delim\mathligtoks\mathlig@delim\relax\else
  \expandafter\checker\mathligtoks\mathlig@delim\fi
  \mathlig@next
 }%
 \mathlig@defcs{mathlig@toks@#1#2}{}%
 \mathlig@defcs{mathlig@chck@#1#2}##1##2\mathlig@delim{%
  \ifx\mathlig@next@tok##1%
   \mathlig@let@cs\mathlig@next@cmd{mathlig@look@#1#2##1}\let\mathlig@next\mathlig@gobble
  \fi
  \ifx\mathlig@delim##2\mathlig@delim\relax\else
   \csname mathlig@chck@#1#2\endcsname##2\mathlig@delim
  \fi
 }%
 \ifx\mathlig@delim#2\mathlig@delim\else
  \mathlig@defcs{mathlig@back@#1#2}{\csname mathlig@back@#1\endcsname #2}%
 \fi
}%

\catcode`@\mathlig@atcode

\mathchardef\ordinarycolon\mathcode`\:
\def\vcentcolon{\mathrel{\mathop\ordinarycolon}}
\mathlig{:=}{\vcentcolon=}
\mathlig{::=}{\vcentcolon\vcentcolon=}

\numberwithin{equation}{section}

\theoremstyle{plain}
\newtheorem{theo}{Theorem}[section]
\newtheorem{cor}[theo]{Corollary}
\newtheorem{lem}[theo]{Lemma}
\newtheorem{prop}[theo]{Proposition}

\theoremstyle{definition}

\newtheorem*{theorem*}{Theorem}
\newtheorem*{idea*}{Idea}

\theoremstyle{remark}
\newtheorem{rem}[theo]{Remark}

\newtheorem*{ex*}{Example}
\newtheorem*{exs*}{Examples}
\newtheorem*{rems*}{Remarks}

\title[Dirac operator with singular interactions on star-graphs]{Self-adjointness of the
2D Dirac operator with singular interactions supported on star-graphs}

\author{Dale~Frymark*}
\address{Department of Theoretical Physics, Nuclear Physics Institute, Czech Academy of Sciences, 25068 Řež, Czech Republic
}
\email{frymark@ujf.cas.cz}

\author{Vladimir Lotoreichik}
\address{Department of Theoretical Physics, Nuclear Physics Institute, Czech Academy of Sciences, 25068 Řež, Czech Republic      
}
\email{lotoreichik@ujf.cas.cz}

\thanks{* Corresponding Author}


\begin{document}

\begin{abstract}
We consider the two-dimensional Dirac operator with Lorentz-scalar $\delta$-shell interactions on each edge of a star-graph. An orthogonal decomposition is performed which shows such an operator is unitarily equivalent to an orthogonal sum of half-line Dirac operators with off-diagonal Coulomb potentials. This decomposition reduces the computation of the deficiency indices to determining the number of eigenvalues of a one-dimensional spin-orbit operator in the interval $(-1/2,1/2)$. 

If the number of edges of the star graph is two or three, these deficiency indices can then be analytically determined for a range of parameters. For higher numbers of edges, it is possible to numerically calculate the deficiency indices. Among others, examples are given where the strength of the Lorentz-scalar interactions directly change the deficiency indices while other parameters are all fixed and where the deficiency indices are $(2,2)$, neither of which have been observed in the literature to the best knowledge of the authors. For those Dirac operators which are not already self-adjoint and do not have $0$ in the spectrum of the associated spin-orbit operator, the distinguished self-adjoint extension is also characterized. 
\end{abstract}

\maketitle

\setcounter{tocdepth}{1}
\tableofcontents

\section{Introduction}

Dirac operators with singular interactions were first studied in the 1980's by Gesztesy and Šeba in one dimension and later by Dittrich, Exner and Šeba \cite{DES, GS} in three dimensions. After that, the subject fell into relative obscurity for 25 years until a revival took place in 2014 due to a series of papers by Arrizabalaga, Mas and Vega~\cite{AMS, AMS2, AMS3}. Since then, Dirac operators with various types of singular interactions have garnered much interest and been studied in many different settings, see e.g. \cite{BEHL1,BEHL2,BHOBP,BFSVDB,HOBP}. This renewed interest is due in part to their applications in physics to the confinement of quarks~\cite{J75} and to the mathematical model of graphene~\cite{BFSVDB2}, and in part because they serve as an idealized model for the Dirac operator with a regular potential localized in the vicinity of a surface~\cite{CLMT, MP}. 

In the present paper, we study the two-dimensional Dirac operator with a singular interaction supported on a star-graph with finitely many edges. This star-graph separates the Euclidean plane into finitely many sectors
with corners at the origin and thus can be viewed as a special corner domain.

Many results have been obtained in the past for polygons. Let $n\in\NN_0$ be the number of non-convex corners of
	the polygon $\Omega\subset\RR^2$; angles with magnitude larger than $\pi$. For instance, Birman and Skvortsov~\cite{BS62} (see also~\cite[Section 10]{DM15} and~\cite{P13}) analyzed the Dirichlet Laplacian acting in the Hilbert space $L^2(\Omega)$ as
\[
	H^2(\Omega)\cap H^1_0(\Omega)\ni u\mapsto -\Delta u,
\]
and found the deficiency indices to be $(n,n)$. A detailed analysis of this and related models in a slightly different terminology can be found in the monographs~\cite{D88, G, G92}.

A direct counterpart of this result for the two-dimensional Dirac operator on a (curvilinear) polygon with the infinite mass boundary condition is obtained by Le Treust and Ourmi\`{e}res-Bonafos in~\cite{TOB} and by Pizzichillo and Van Den Bosch in~\cite{PVDB}; the deficiency indices again given by the number of non-convex corners. The manuscript \cite{PVDB} also obtains a generalization for Lorentz-scalar $\delta$-shell interactions supported on the boundary of a polygon, with the deficiency indices given simply by the number of corners.

Behind all these considerations of Dirac operators on polygons there is an explicit analysis of the Dirac operator on an infinite sector combined with a localization technique. This model of an infinite sector, however, admits further generalizations. One important case is considered by Cassano and the second author~\cite{CL} for infinite mass boundary conditions: where the orientation of the normal vector is opposite on the two edges of the sector. In that case, the normal vector enters the boundary condition and gives rise to a slightly different model where the deficiency indices are unequal. 

To the best knowledge of the authors, no previous investigation of two-dimensional Dirac operators on sectors with singular interactions
have revealed deficiency indices higher than $(1,1)$. Moreover, there are no known cases where varying the singular interaction strength and fixing all geometric parameters results in the deficiency indices changing. It is therefore natural to wonder whether two-dimensional Dirac operators with singular interactions on star-graphs, a further generalization of the infinite sector model discussed above, also adhere to these trends -- these lines of inquiry are the main motivations of the manuscript. The answers, determined in Subsections \ref{ss-threeequalleads} and \ref{ss-22}, are negative; there are star-graphs where the deficiency indices of the respective Dirac operator vary with the strengths of the singular interactions and where they are $(2,2)$.  

In order to discuss the self-adjointness of such Dirac operators and other results of the manuscript, it is first necessary to introduce some terminology. Recall that the $2\times 2$ Hermitian {\em Pauli matrices} $\sigma_1$, $\sigma_2$, and $\sigma_3$ are given by
\begin{align*}
\sigma_1=\begin{pmatrix}
0 & 1 \\
1 & 0
\end{pmatrix},\hspace{1.5em}
\sigma_2=\begin{pmatrix}
0 & -i \\
i & 0
\end{pmatrix}\text{ and }\hspace{.5em}
\sigma_3=\begin{pmatrix}
1 & 0 \\
0 & -1
\end{pmatrix}.
\end{align*}
These matrices satisfy the anti-commutation relation $\sigma_j\sigma_i+\sigma_i\sigma_j=2\delta_{ij}$ for $i,j\in\{1,2,3\}$, where $\delta_{ij}$ is the Kronecker delta. It is convenient to define $\sigma:=(\sigma_1,\sigma_2)$ so that for $\va{x}=(x_1,x_2)^{\top}\in\RR^2$
\begin{align*}
\sigma\cdot \va{x}:=x_1\sigma_1+x_2\sigma_2=
\begin{pmatrix} 
0 & x_1-ix_2 \\
x_1+ix_2 & 0 
\end{pmatrix}.
\end{align*}
The first-order Dirac differential expression in $\RR^2$ is given by
\begin{align}\label{e-freedirac}
\cD:=-i(\sigma\cdot\nabla)=\begin{pmatrix}
0 & -i(\partial_1-i\partial_2) \\
-i(\partial_1+i\partial_2) & 0
\end{pmatrix}.
\end{align}
A mass term is not included in the expression as this is crucial only for spectral considerations which fall outside the scope of the current manuscript. Let $\Gamma = \cup_{j=1}^N\Gamma_j$ be a star-graph, where $\{\Gamma_j\}_{j=1}^N$ are infinite rays emerging from the origin with normal vectors all oriented clockwise and enumerated counter-clockwise; thereby separating $\RR^2$ into $N$ sectors. Let $\{\tau_j\}_{j=1}^N\in\RR^N$ denote the strengths of Lorentz-scalar $\delta$-shell interactions on each respective edge of the star-graph $\Gamma$.
We consider the Dirac operator $\fD_N$, which is associated to the formal differential expression
\begin{equation}\label{eq:diff_expr}
	-i(\sigma\cdot\nabla) + \sum_{j=1}^N \tau_j\sigma_3\delta_{\Gamma_j},
\end{equation}   
where $\delta_{\Gamma_j}$ is the Dirac $\delta$-function supported on $\Gamma_j$.
The operator $\fD_N$ is rigorously defined
in the Hilbert space $L^2(\RR^2,\CC^2)$
acting via the differential expression $\cD$ on the domain consisting of two-component $H^1$-functions on each sector subject to special boundary conditions on each edge $\Gamma_j$ corresponding to $\tau_j$~\cite{BHOBP, CLMT} (see Section~\ref{s-star} for details).
Throughout the paper we exclude the case of confinement and assume that $\tau_j\ne \pm 2$ for all $j\in\{1,2,\dots,N\}$. It should be noted that formally equation \eqref{eq:diff_expr} anticommutes with the antiunitary involution given by the composition of componentwise complex conjugation and multiplication by $\sigma_1$.

The operator $\fD_N$ is shown to be closed, densely defined and symmetric, so it remains only to identify whether it has non-trivial deficiency indices or is self-adjoint. Our method of analysis relies on separation of variables in polar coordinates, with the angular component naturally giving rise to a one-dimensional spin-orbit operator $J_N$ acting in the Hilbert space $L^2(\mathbb{S}^1,\CC^2)$ with $N$ special point interactions. This spin-orbit operator is shown to be self-adjoint by identifying it with a momentum-type operator on a certain graph~\cite{E}. 

Spectral analysis of the spin-orbit operator $J_N$ allows us to obtain a decomposition of $\fD_N$ into an orthogonal sum of half-line Dirac operators with off-diagonal Coloumb potentials, expressed via the eigenvalues of $J_N$. In particular, we find that the multiplicity of the eigenvalues of $J_N$ is at most two; double eigenvalues should be treated with special care in the decomposition of $\fD_N$. Further subtlety is required to handle a potential double eigenvalue at $0$. 

Fortunately, the half-line Dirac operators arising in the decomposition of $\fD_N$ have appeared many times in the literature and with the help of, i.e.~\cite{CP}, it is shown that the deficiency indices $(n_+(\fD_N),n_-(\fD_N))$ of $\fD_N$ are precisely equal to half of the number of the eigenvalues of $J_N$ in the interval $(-1/2,1/2)$, with multiplicities taken into account. 
Thus, determining the deficiency indices of $\fD_N$ reduces to the spectral analysis of $J_N$. In the most general case it is only possible to prove an upper bound using perturbation theory:
\begin{equation}\label{eq:bnd}
n_+(\fD_N) = n_-(\fD_N) \le N.
\end{equation}

However, in many special cases it is possible to conduct an explicit analysis of the spectrum of $J_N$. The most notable being when $N = 2$, i.e.~the singular interactions are supported on a broken line, see Section \ref{s-wedge} for full details.  Several conditions are given in Proposition \ref{c-simple} that help classify when eigenvalues are simple or have multiplicity two. 
If the interaction  strengths are equal on each edge, so $\tau_1 = \tau_2$, then $n_+(\fD_2) = n_-(\fD_2) = 1$ unless $\Gamma$ is a straight line, where the respective Dirac operator is self-adjoint. Note this setup is essentially covered by the analysis in~\cite{PVDB}. If the interaction strengths are opposite on each edge, so $\tau_1 = -\tau_2$, we also obtain that $n_+(\fD_2) = n_-(\fD_2) = 1$ (even in the case $\Gamma$ is a straight line). The additional degenerate case where the interaction supported on one of the rays has strength $0$, and therefore the geometric situation collapses to the consideration of an interaction supported on one ray, also has $n_+(\fD_2) = n_-(\fD_2) = 1$. To the best of our knowledge, it is the first time that a Lorentz-scalar $\delta$-shell interaction supported on a non-closed curve is analyzed.

Explicit analysis is also viable when $\Gamma$ is a symmetric graph with $N = 3$ edges and $\tau = \tau_1 = \tau_2 = \tau_3\ne \pm 2$. Critically, it is observed in Corollary \ref{c-3defindices} that the respective Dirac operator transitions from being self-adjoint to having deficiency indices $(1,1)$ when the absolute value of the interaction strength $\tau$ exceeds $2\sqrt{3}$. 

Eigenvalues of $J_N$ for $N=2,3$ were, in essence, established by taking the determinant of a $2N\times2N$ matrix and setting it equal to zero. The resulting equation involves the angles and singular interaction strengths of each edge, as well as the spectral parameter. For star-graphs with a higher number of edges, it is therefore not feasible to obtain and then extract spectral information from such a formula. Thankfully, it is possible to establish an alternative spectral condition, at least for symmetric star-graphs, thanks to the identification of $J_N$ with a momentum operator on a certain graph. In particular, the number of eigenvalues of $J_N$ in the interval $(-1/2,1/2)$ with multiplicities taken into account is found to be equal to the number of eigenvalues lying on the arc $\{z\in\TT\colon 0< {\rm arg}(z)<2\pi/N\}$ of the unit circle, with multiplicities, of an explicitly given unitary matrix. This method is applied to the symmetric star-graph with $N=6$ edges and coupling constants $\tau_1 = \tau_3 = \tau_4 =\tau_6 = 1$, $\tau_2=\tau_5 = -1$ in order to show that the respective Dirac operator possesses deficiency indices $(2,2)$. Further numerical tests using the alternative condition have not identified any examples with deficiency indices higher than $(1,1)$ when $N\le 5$, indicating that the bound in equation \eqref{eq:bnd} can probably be improved. We also expect that by taking star-graphs with large number of edges and properly adjusted interaction strengths one can construct examples with arbitrarily large equal deficiency indices.

It is also possible to consider electrostatic $\delta$-shell interactions on each lead of the $N$ star-graph by adding extra contributions to equation \eqref{eq:diff_expr}, as the associated Dirac operator $\fD_N$ is still densely defined and symmetric. Indeed, such interactions are initially considered in Section \ref{s-wedge} but the subsequent analysis immediately breaks down because the respective spin-orbit operator is not self-adjoint. Computation of the deficiency indices for such interactions remains an open question. 

Finally, when the deficiency indices of $\fD_N$ are non-trivial we characterize self-adjoint extensions of $\fD_N$ using the classical von Neumann extension theory. If $0\notin\sigma(J_N)$, then it is possible to identify the unique ``distinguished'' self-adjoint extension; the extension whose domain is contained in the Sobolev space $H^{1/2}$ on each sector. In fact, the regularity of this extension is slightly better, in terms of the scale of Sobolev spaces, and can be computed in terms of the spectral gap for $J_N$. The special case $0\in\sigma(J_N)$, on the other hand, does not have a self-adjoint extension whose domain in contained in $H^{1/2}$ and there is no clear way to fix the distinguished self-adjoint extension.

\subsection*{Structure of the paper}\label{ss-structure}

Section \ref{s-wedge} presents a detailed analysis of the case where the star-graph has $N=2$ edges -- where the star-graph is simply a broken line. This analysis provides both a base case for the more general star-graphs of Section \ref{s-star} as well as many extra details that are too difficult to calculate in full generality. Subsections \ref{ss-twoleads} and \ref{ss-twoleads2} determine the number of eigenvalues of the spin-orbit operator in the interval $(-1/2,1/2)$ when the strengths of the Lorentz-scalar interactions are equal or opposite on either edge. In Subsection~\ref{ss-slit} we perform the same analysis for the Lorentz-scalar interaction supported on a ray. Subsection \ref{ss-orthowedge} then details an orthogonal decomposition which shows that the Dirac operator is unitarily equivalent to half-line Dirac operators with off-diagonal Coulomb potentials which depend on the studied eigenvalues, see Theorem \ref{p-decomp}.

Section \ref{s-star} proves that the orthogonal decomposition of the Dirac operator with $N=2$ edges from Subsection \ref{ss-orthowedge} holds in general for $N\in\NN$. In particular, the proof that the spin-orbit operator is self-adjoint is extended and relies on an important identification with a momentum operator on a certain graph, see Proposition \ref{p-sathree}. Subsection \ref{ss-threeleads} gives an explicit formula for the eigenvalues of the spin-orbit operator when the star graph consists of $N=3$ edges. Subsection \ref{ss-threeequalleads} then uses this formula to analyze the special case where the three edges are equally spaced and the Lorentz-scalar $\delta$-shell interactions on each lead have the same strength. In the special case where the sectors of $\RR^2$ defined by the star-graph each have the same opening angle, Subsection \ref{ss-altspec} exploits the connection to momentum operators used in some of the previous proofs to reduce the problem of finding eigenvalues of the spin-orbit operator in an interval to finding eigenvalues of a unitary matrix on an arc of the unit circle. An example using this method is computed in Subsection \ref{ss-22} where $N=6$ and the respective Dirac operator is shown to have deficiency indices $(2,2)$.

Dirac operators that are not already self-adjoint then have their self-adjoint extensions characterized by the classical von Neumann theory in Section \ref{s-distinguished}. When $0\notin\sigma(J_N)$, the unique distinguished self-adjoint extension is also constructed. Appendix \ref{app} includes a technical lemma which is used in the proof of the orthogonal decomposition.

\section{Singular interaction supported on a broken line}\label{s-wedge}

We first consider the special case where the star-graph $\Gamma$ has only two edges and therefore splits the plane into two sectors. Adhering to the nomenclature of existing literature, such a star-graph is viewed as a broken line. While a similar situation is considered in~\cite{PVDB}, where the Lorentz-scalar $\delta$-shell interaction is supported on the boundary of a curvilinear polygon, the analysis here differs significantly. In particular, we assume that the Lorentz-scalar interactions can be of different strengths on each edge and include a partial analysis of the case where both electrostatic and Lorentz-scalar $\delta$-shell interactions are allowed. Also note that the notation of this section differs from that of Section \ref{s-star}, as it is more suitable for this simple geometric configuration.

Define the wedge
\begin{align*}
    \Omega_+:=\left\{(r\cos\te,r\sin\te)\in\RR^2\colon r>0,~\te\in\II_+\right\}\subset\RR^2,
\end{align*}
where $\II_+:=(-\w,\w)$ with $\w\in(0,\pi)$. The wedge has an opening angle of $2\w$ and divides the plane into two domains; the complementary sector is then denoted by 
\[
\Omega_-:=\RR^2\setminus\overline{\Omega}_+
= \left\{(r\cos\te,r\sin\te)\in\RR^2\colon r>0,~\te\in\II_-\right\},
\]
where $\II_-:=(\w,2\pi-\w)$.
A function (spinor) $u\in L^2(\RR^2,\CC^2)$ can be decomposed into two parts corresponding to these domains
\begin{align*}
    u=u_+\oplus u_-\in L^2(\Omega_+,\CC^2)\oplus L^2(\Omega_-,\CC^2),~~\text{ where }~~ u_\pm :=u|_{\Omega_\pm}.
\end{align*}
The two sides of the wedge can then be denoted by
\begin{align*}
\Gamma^l&:=\left\{(r\cos\w,r\sin\w)\in\RR^2\colon r>0\right\}, ~\text{and} \\
\Gamma^r&:=\left\{(r\cos\w,-r\sin\w)\in\RR^2\colon r>0\right\}.
\end{align*}
The dependence on $\w$ is suppressed for the sake of notational accessibility and the choice $\w=\pi/2$, of course, refers to the straight line. The Dirac operator of interest may possess two different types of singular interactions on each edge: electrostatic and Lorentz-scalar $\delta$-shell interactions. The strengths of these interactions will be interpreted by the real parameters $\eta_l$, $\eta_r$ and $\tau_l$, $\tau_r$, respectively. Letting $\sigma_0$ denote the identity matrix of rank two, such a Dirac operator corresponds to the formal differential expression:
\begin{align}\label{eq:formal}
    -i(\sigma\cdot\nabla)+(\eta_l\sigma_0+\tau_l\sigma_3)\delta_{\Gamma^l}+(\eta_r\sigma_0+\tau_r\sigma_3)\delta_{\Gamma^r}.
\end{align}
The interactions can be translated into boundary conditions on each edge of the wedge, and were first rigorously described in \cite{BHOBP}. Denote the normal vectors on $\Gamma^l$ and $\Gamma^r$, which are chosen to be pointing outward from $\Omega_+$, as $\va{\nu}_l=(\nu_{l,1},\nu_{l,2})=(-\sin\w,\cos\w)$ and $\va{\nu}_r=(\nu_{r,1},\nu_{r,2})=(-\sin\w,-\cos\w)$, respectively. Recall that the free massless Dirac differential expression is denoted by $\cD$ and given in equation \eqref{e-freedirac}. The operator associated to the formal differential expression~\eqref{eq:formal} in the current context is rigorously given by
\begin{equation}\label{e-dirac}
    \begin{aligned}
    \fD u&\!:=\! (\cD u_+)\oplus(\cD u_-), \\
    \dom\fD&\!:=\!\Big\{u=u_+\oplus u_-\in H^1(\Omega_+,\CC^2)\oplus H^1(\Omega_-,\CC^2)\colon \\
    &-i(\sigma_1\nu_{j,1}+\sigma_2\nu_{j,2})\left(u_+|_{\Gamma^j}\!-\!u_-|_{\Gamma^j}\right)\!=\!\frac{1}{2}(\eta_j\sigma_0\!+\!\tau_j\sigma_3)\left(u_+|_{\Gamma^j}+u_-|_{\Gamma^j}\right),j\!=\!l,r\Big\}. \\
\end{aligned}
\end{equation}
The trace operator implicitly denoted by the ``restriction'' in equation \eqref{e-dirac} is bounded as an operator from $H^1(\Omega_{\pm},\CC^2)$ into $H^{1/2}(\partial\Omega_{\pm},\CC^2)$ via \cite[Theorem 3.38]{M}, the proof of which still applies in this case because the regions $\Omega_{\pm}$ are Lipschitz hypographs.

While it is possible to consider the operator $\fD$ for all real interaction strengths, we make the additional assumption that neither edge is in the case of confinement: when $\eta_l^2-\tau_l^2=-4$ or $\eta_r^2-\tau_r^2=-4$. If both of these conditions are met, it is possible to decouple the operator into orthogonal sum of operators with respect to decomposition  $L^2(\RR^2,\CC^2) = L^2(\Omega_+,\CC^2)\oplus L^2(\Omega_-,\CC^2)$, see \cite[Lemma 4.1]{BHOBP}. In other words, the boundary conditions can be rewritten as pertaining to each sector individually. Therefore, we prefer to work in the non-confining regime. The intermediate case of one confining edge and one non-confining requires an additional consideration that goes beyond the present paper.

The values $\eta_l^2-\tau_l^2$ and $\eta_r^2-\tau_r^2$ will consistently play a role in calculations, and so will be denoted by $\eps_l$ and $\eps_r$, respectively. Furthermore, it is convenient to label
\begin{align}\label{e-pm}
p_l=1+\frac{1}{4}\eps_l \hspace{.5em} \text{ and } \hspace{.5em} m_l=1-\frac{1}{4}\eps_l,
\end{align}
with $p_r$ and $m_r$ being similarly defined.

It is also necessary to introduce the operator using standard polar coordinates. Let $\va{x}=(x_1,x_2)$ so that $x_1=r\cos\te$ and $x_2=r\sin\te$ and define
\begin{align*}
    \va{e}\ti{rad}(\te)=\frac{\partial \va{x}}{\partial r}=\begin{pmatrix} \cos\te \\ \sin\te \end{pmatrix} \hspace{.5em} \text{ and } \hspace{.5em} 
    \va{e}\ti{ang}(\te)=\frac{\partial \va{e}\ti{rad}}{\partial \te}=\begin{pmatrix} -\sin\te \\ \cos\te \end{pmatrix}.
\end{align*}
The Hilbert spaces $L^2\ti{pol}(\Omega_\pm,\CC^2):=L^2(\RR_+\times\II_\pm,\CC^2;rdrd\te)$ can then be viewed as the tensor products $L^2_r(\RR_+)\otimes L^2(\II_\pm,\CC^2)$, with $L^2_r(\RR_+)=L^2(\RR_+;rdr)$. The unitary transformation
\begin{align*}
    V_\pm:L^2(\Omega_\pm,\CC^2)\to L^2\ti{pol}(\Omega_\pm,\CC^2),
    \qquad (V_\pm u)(r,\theta) := u(r\cos\theta,r\sin\theta), 
\end{align*}
is responsible for changing coordinate systems and yields the polar Sobolev spaces
\begin{align*}
    H\ti{pol}^1(\Omega_\pm,\CC^2):=V_\pm(H^1(\Omega_\pm,\CC^2))=\left\{u\colon u,\partial_r u, r^{-1}(\partial_{\te} u)\in L^2\ti{pol}(\Omega_\pm,\CC^2)\right\}.
\end{align*}
The Dirac operator acting on $L^2\ti{pol}(\Omega_+,\CC^2)\oplus L^2\ti{pol}(\Omega_-,\CC^2)$ is therefore given by
\begin{align*}
    \widetilde{\fD}:=V\fD V^{-1}, \qquad \dom\widetilde{\fD}:=V(\dom\fD),
\end{align*}
where $V := V_+\oplus V_-$.
The corresponding Dirac differential expression in polar coordinates is now
\begin{align}\label{e-polaraction}
    \widetilde{\cD}u=-i(\sigma\cdot\va{e}\ti{rad})\left(\partial_r u+\frac{u}{2r}-\frac{\left(-i\sigma_3\partial_{\te}+\frac{1}{2}\right)u}{r}\right).
\end{align}
In analogy to equation \eqref{e-dirac}, the Dirac operator of interest is thus
\begin{equation}\label{e-polardirac}
    \begin{aligned}
    \widetilde{\fD} u&\!:=\! (\widetilde{\cD} u_+)\oplus(\widetilde{\cD} u_-), \\
    \dom\widetilde{\fD}&\!:=\!\big\{u=u_+\oplus u_-\in H^1\ti{pol}(\Omega_+,\CC^2)\oplus H^1\ti{pol}(\Omega_-,\CC^2)~:~ \\
    &-i(\sigma_1\nu_{j,1}+\sigma_2\nu_{j,2})\left(u_+|_{\Gamma^j}\!-\!u_-|_{\Gamma^j}\right)=\frac{1}{2}(\eta_j\sigma_0+\tau_j\sigma_3)\left(u_+|_{\Gamma^j}\!+\!u_-|_{\Gamma^j}\right), j \!=\! l,r\big\}.
    \end{aligned}
\end{equation}

A more convenient form of the boundary conditions imposed on both $\dom\fD$ and $\dom\widetilde{\fD}$ can be immediately obtained by using \cite[Lemma 4.1]{BHOBP}.

\begin{lem}\label{p-betterbcs}
If the non-confining case where $\eps_l,\eps_r\neq-4$ holds, then $u=u_+\oplus u_-\in H^1_{\rm pol}(\Omega_+,\CC^2)\oplus H^1_{\rm pol}(\Omega_-,\CC^2)$
belongs to $\dom\fD$ if and only if
\begin{equation}\label{e-betterbcs}
\begin{aligned}
u_-|_{\Gamma^l}&=\frac{1}{p_l}
\begin{pmatrix}
m_l & -e^{-i\w}(\eta_l-\tau_l) \\
e^{i\w}(\eta_l+\tau_l) & m_l
\end{pmatrix}
u_+|_{\Gamma^l}, \\
u_-|_{\Gamma^r}&=\frac{1}{p_r}
\begin{pmatrix}
m_r & e^{i\w}(\eta_r-\tau_r) \\
-e^{-i\w}(\eta_r+\tau_r) & m_r
\end{pmatrix}
u_+|_{\Gamma^r}.
\end{aligned}
\end{equation}
\end{lem}

The matrices from the above lemma that give the relation between the boundary values will be denoted by $M_l$ and $M_r$, respectively. These matrices do not possess many nice properties, like in the cases of infinite-mass boundary conditions or only Lorentz-scalar interactions, see e.g.~\cite[Proposition 2.2]{PVDB}. This is the effect of having both $\sigma_0$ and $\sigma_3$ in the boundary conditions, which commute and anti-commute with $\sigma\cdot\va{\nu}$, respectively. However, some quick calculations show that 
\begin{equation}\label{e-adjointsrelation}
\begin{aligned}
(M_l)^*(\sigma\cdot\va{\nu}_l)(M_l)&=(\sigma\cdot\va{\nu}_l), \\
(M_r)^*(\sigma\cdot\va{\nu}_r)(M_r)&=(\sigma\cdot\va{\nu}_r).
\end{aligned}
\end{equation}
These relations are sufficient to show that the operator is symmetric.

\begin{prop}\label{p-symmetric}
The operator $\fD$ is densely defined and symmetric in the Hilbert space $L^2(\Omega_+,\CC^2)\oplus L^2(\Omega_-,\CC^2)$.
\end{prop}

\begin{proof}
The operator is densely defined in $L^2(\Omega_+,\CC^2)\oplus L^2(\Omega_-,\CC^2)$ because $C_0^{\infty}(\Omega_\pm,\CC^2)\subset\dom\fD$ is dense in $L^2(\Omega_\pm,\CC^2)$. Following \cite[Proposition 2.1]{CL}, both $\Omega_+$ and $\Omega_-$ are epigraphs of globally Lipschitz functions so it is easy to derive from \cite[Theorem 3.34]{M} that the Green's identity
\begin{align*}
    \int_{\Omega_\pm}(\cD u_{\pm})\cdot\overline{v_{\pm}}dx-\int_{\Omega_\pm}u_{\pm}\cdot\overline{(\cD v_{\pm})}dx=\mp i\int_{\Gamma^l\cup\Gamma^r}((\sigma\cdot\va{\nu})u_{\pm})\cdot\overline{v_{\pm}}ds
\end{align*}
holds for all $u,v\in H^1(\Omega_\pm,\CC^2)$ where $ds$ denotes the arc length measure along the boundary and the normal $\va{\nu}$ is outwards facing from $\Omega_+$. Hence, for $u,v\in\dom\fD$ we have
\begin{align}\label{e-greens}
    \langle \fD u,v\rangle-\langle u,\fD v\rangle=-i\left[\int_{\Gamma^l\cup\Gamma^r}((\sigma\cdot\va{\nu})u_+)\cdot\overline{v_+}ds
    -\int_{\Gamma^l\cup\Gamma^r}((\sigma\cdot\va{\nu})u_-)\cdot\overline{v_-}ds\right],
\end{align}
where the inner products on the left-hand side are taken in $L^2(\Omega_+,\CC^2)\oplus L^2(\Omega_-,\CC^2)$. The boundary conditions from equation \eqref{e-betterbcs} then allow $u_-|_{\Gamma^l}$, $u_-|_{\Gamma^r}$ and $v_-|_{\Gamma^l}$, $v_-|_{\Gamma^r}$ to be rewritten and we see that
\begin{align*}
    \int_{\Gamma^l\cup\Gamma^r}((\sigma\cdot\va{\nu})u_-)\cdot\overline{v_-}ds&=\int_{\Gamma^l}(\sigma\cdot\va{\nu}_l)(M_lu_+)\cdot\overline{(M_lv_+)}ds 
    +\int_{\Gamma^r}(\sigma\cdot\va{\nu}_r)(M_ru_+)\cdot\overline{(M_rv_+)}ds \\
    &=\int_{\Gamma^l}M_l^*(\sigma\cdot\va{\nu}_l)(M_lu_+)\cdot\overline{(v_+)}ds 
    +\int_{\Gamma^r}M_r^*(\sigma\cdot\va{\nu}_r)(M_ru_+)\cdot\overline{(v_+)}ds \\
    &=\int_{\Gamma^l\cup\Gamma^r}((\sigma\cdot\va{\nu})u_+)\cdot\overline{(v_+)}ds
\end{align*}
with the last equality utilizing equation \eqref{e-adjointsrelation}. The right-hand side of equation \eqref{e-greens} then vanishes and we conclude that $\fD$ is a symmetric operator.
\end{proof}

Clearly, this also means that $\widetilde{\fD}$ is also densely defined and symmetric in $L^2\ti{pol}(\Omega_+,\CC^2)\oplus L^2\ti{pol}(\Omega_-,\CC^2)$. Equation \eqref{e-polaraction} also motivates the definition of an associated spin-orbit-type operator which will act in the Hilbert space $L^2(\II_+,\CC^2)\oplus L^2(\II_-,\CC^2)$ for simplicity:
\begin{equation}\label{e-angular}
\begin{aligned}
J\f&:=\left(-i\sigma_3\f_+'+\frac{\f_+}{2}
\right)\oplus\left(-i\sigma_3\f_-'+\frac{\f_-}{2}
\right), \\
\dom J&:=\Big\{\f=\f_+\oplus\f_-\in H^1(\II_+,\CC^2)\oplus H^1(\II_-,\CC^2)\colon \\
&\hspace{12em}
\f_-(\omega)=M_l
\f_+(\omega),~ 
\f_-(2\pi-\omega)=M_r
\f_+(-\omega)\Big\}.
\end{aligned}
\end{equation}

\begin{prop}\label{p-momentum2}
Assume that $\eps_l,\eps_r\neq\pm4$. The operator $J$ is self-adjoint if and only if $\eta_l=\eta_r=0$. In particular, $J$ has compact resolvent when $\eta_l=\eta_r=0$.
\end{prop}

\begin{proof}
The operator $J-1/2$ can be viewed as a momentum operator on a graph with four directed edges, with vectors denoted $\psi_j$ for $j\in\{1,\dots,4$\}, and two vertices: $-\w$ and $\w$. The first and fourth edges are oriented from $-\w$ to $\w$ and the second and third are reversed. The length of the first and the second edges is $2\omega$, while the length of the remaining two edges is $2\pi-2\omega$. In this way, $\psi_1$ and $\psi_2$ represent the two components of $\f_+$ and $\psi_3$ and $\psi_4$ represent the two components of $\f_-$. The Hamiltonian acts on these vectors as $-i\frac{d}{dx}$ and the effect of $\sigma_3$ in $J-1/2$ is accounted for by the direction of the edges. Define the vectors
\begin{align*}
    \psi\ti{out}&:=(\psi_1(-\w),\psi_2(\w),\psi_3(\w),\psi_4(2\pi-\w)), \\
    \psi\ti{in}&:=(\psi_1(\w),\psi_2(-\w),\psi_3(2\pi-\w),\psi_4(\w)).
\end{align*}
Note that the matrix $M_r$ is identical for the angles $-\w$ and $2\pi-\w$, which allows for the vertices $-\w$ and $2\pi-\w$ to be identified and simply denoted by $-\w$. Self-adjoint realizations of this Hamiltonian are in one-to-one correspondence with $4\times 4$ unitary matrices $\cU$ via the boundary condition $\cU\psi\ti{in}=\psi\ti{out}$ by \cite[Proposition 4.1]{E}. 

Note the helpful identity 
\begin{align*}
    \frac{m_j^2+\eps_j}{m_jp_j}=\frac{p_j}{m_j},
\end{align*}
for $j=l,r$. The form of $\cU$ can then be found by rewriting the boundary conditions in equation \eqref{e-betterbcs} and solving for the appropriate boundary values
\begin{align*}
    \cU=\begin{pmatrix}
    0 & -\frac{e^{i\w}(\eta_r-\tau_r)}{m_r} & \frac{p_r}{m_r} & 0 \\
    -\frac{e^{i\w}(\eta_l+\tau_l)}{m_l} & 0 & 0 & \frac{p_l}{m_l} \\
    \frac{p_l}{m_l} & 0 & 0 & -\frac{e^{-i\w}(\eta_l-\tau_l)}{m_l} \\
    0 & \frac{p_r}{m_r} & -\frac{e^{-i\w}(\eta_r+\tau_r)}{m_r} & 0
    \end{pmatrix}.
\end{align*}
Calculations of $\cU\cU^*$ and $\cU^*\cU$, which we omit for brevity, reveal that $\cU$ is unitary if and only if the following four equations are satisfied for both $j=l,r$, for a total of eight conditions:
\begin{align}\label{e-unitary1}
    \left[\frac{(\eta_j\pm\tau_j)}{m_j}\right]^2+\left[\frac{p_j}{m_j}\right]^2&=1,\nonumber\\
   (\eta_j\pm\tau_j)+(\eta_j\mp\tau_j)&=0.
\end{align}
Equation \eqref{e-unitary1} is clearly satisfied for each index only when $\eta_j=0$. For $\eta_j=0$, the operator $J$ then has compact resolvent by \cite[Theorem~5.1]{E}. The addition of a constant $1/2$ to the Hamiltonian has no effect on these properties and the result follows. 
\end{proof}

The spectral decomposition of $J$ will be used to separate variables for $\fD$, to determine its deficiency indices and to parametrize its self-adjoint extensions.	Self-adjointness of $J$ is therefore crucial for the analysis.
The previous proposition does not claim that for $\eta_l,\eta_r\neq0$ there are no self-adjoint extensions of the Dirac operator $\fD$ given in equation \eqref{e-dirac}. However, this specific choice for the spin-orbit operator does not lead to a viable separation of variables, which is in itself somewhat surprising. This is also the case for the general star-graph with $N$ edges and can be observed if the proof of Proposition \ref{p-momentum2} is adapted to the setup of Proposition \ref{p-sathree}. We therefore restrict our attention in the rest of the manuscript to the case where only Lorentz-scalar $\delta$-shell interactions are present.

Normalized eigenfunctions of the operator $J$ now form an orthonormal basis of the Hilbert space $L^2(\SS^1,\CC^2) = L^2(\II_+,\CC^2)\oplus L^2(\II_-,\CC^2)$. Let $\f=(\f_{+,1},\f_{+,2})\oplus(\f_{-,1},\f_{-,2})\in\dom J$ and $\widetilde{\la}\in\RR$ satisfy $J\f=\widetilde{\la}\f$. Then, $\f$ satisfies
\begin{align*}
    &-i\f_{\pm,1}'=(\widetilde{\la}-1/2)\f_{\pm,1}, \\
    &+i\f_{\pm,2}'=(\widetilde{\la}-1/2)\f_{\pm,2}.
\end{align*}
Treatment of these eigenfunctions is made marginally easier by setting $\la:=\widetilde{\la}-1/2$. The generic solution for this system of equations, now using $\la$, is easily seen to be 
\begin{align}\label{e-genericeigenfunction}
    \f_{\pm}=\begin{pmatrix}
    c_{\pm, 1}e^{i\la\te} \\
    c_{\pm, 2}e^{-i\la\te}
    \end{pmatrix},
\end{align}
with the four constants $c_{\pm, 1},c_{\pm, 2}\in\CC$. Plugging these generic eigenfunctions into the boundary conditions from equation \eqref{e-angular} with the relevant values of $\te$ yields the following four equations:
\begin{equation}\label{e-2matrixequations}
    \begin{aligned}
    c_{-,1}&=c_{+,1}\frac{m_l}{p_l}+c_{+,2}\frac{\tau_le^{-i\w(2\la+1)}}{p_l}, \\
    c_{-,2}&=c_{+,1}\frac{\tau_le^{i\w(2\la+1)}}{p_l}+c_{+,2}\frac{m_l}{p_l}, \\
    c_{-,1}&=c_{+,1}\frac{m_re^{-i2\pi\la}}{p_r}-c_{+,2}\frac{\tau_re^{i[\w(2\la+1)-2\pi\la]}}{p_r}, \\
    c_{-,2}&=-c_{+,1}\frac{\tau_re^{-i[\w(2\la+1)-2\pi\la]}}{p_r}+c_{+,2}\frac{m_re^{i2\pi\la}}{p_r}.
\end{aligned}
\end{equation}

Recall the notation given in equation \eqref{e-pm}. This system of equations can then be used to find eigenvalues of the operator $J$.

\begin{prop}\label{p-eigenvalues2}
Let $\tau_l,\tau_r\neq\pm2$, $\w\in(0,\pi)$ and $J$ be given by equation \eqref{e-angular}. Then $\la$ is an eigenvalue of $J-1/2$
if and only if it satisfies the equation
\begin{equation}\label{e-2eigenvalues}
\begin{aligned}
1-\frac{m_lm_r\cos(2\pi\la)}{p_lp_r}-\frac{\tau_l\tau_r\cos{[2\w(2\la+1)-2\pi\la]}}{p_lp_r}=0.
\end{aligned}
\end{equation}
\end{prop}

\begin{proof}
The system of equations given by equation \eqref{e-2matrixequations} has a solution if and only if the following determinant vanishes
\begin{align*}
    \begin{vmatrix}
    1 & 0 & -\frac{m_l}{p_l} & -\frac{\tau_le^{-i\w(2\la+1)}}{p_l} \\
    0 & 1 & -\frac{\tau_le^{i\w(2\la+1)}}{p_l} & -\frac{m_l}{p_l} \\
    1 & 0 & -\frac{m_re^{-i2\pi\la}}{p_r} & \frac{\tau_re^{i[\w(2\la+1)-2\pi\la]}}{p_r} \\
    0 & 1 & \frac{\tau_re^{-i[\w(2\la+1)-2\pi\la]}}{p_r} & -\frac{m_re^{i2\pi\la}}{p_r}
    \end{vmatrix}
    =0.
\end{align*}


A calculation of this determinant yields the desired expression.
\end{proof}

\begin{rem}\label{rem:symmetry}
Note that writing the boundary conditions from equation \eqref{e-betterbcs} in the alternate form $u_+|_{\Gamma^j}=\widetilde{M}_ju_-|_{\Gamma^j}$ for $j=l,r$ and matrix $\widetilde{M}_j$ makes no difference in the calculated determinant. Also, it should be clear that the choice of normal vectors will not influence any results, as the setups are unitarily equivalent after a rotation. It also follows that the spectrum of $J$ for $\w=\wt{\w}\in(0,\pi)$ coincides with that for $\w=\pi-\wt\w$, which reflects the fact that the respective problem is symmetric about $\w=\pi/2$.
\hfill$\diamondsuit$
\end{rem}

\begin{cor}\label{c-2specsymmetric}
For any choice of $\w\in(0,\pi)$ and $\tau_l,\tau_r\neq\pm 2$, the eigenvalues of $J$ are symmetric about the origin.
\end{cor}

\begin{proof}
Rewrite equation \eqref{e-2eigenvalues} in terms of $\widetilde{\la}$ via the substitution  $\la=\widetilde{\la}-1/2$. It is apparent that $\widetilde{\la}$ is a solution to this rewritten equation
if and only if $-\widetilde{\la}$ is a solution for the same values of $\w$, $\tau_l$ and $\tau_r$.
\end{proof}

The matter of determining the deficiency indices for $\fD$ will boil down to determining how many eigenvalues of $J$ lie in the interval $(-1/2,1/2)$ with multiplicities taken into account. Before progressing any further, a detailed examination of the equation \eqref{e-2eigenvalues} is conducted to determine if the number of eigenvalues in this interval can be obtained. It is possible to give a rough upper bound on the number of eigenvalues of $J$ in the interval $(-1/2,1/2)$ in order to help frame this analysis.

\begin{prop}\label{prop:def_bnd}
	Assume that $\tau_l,\tau_r\in\RR\setminus\{-2,0,2\}$
	and $\w\in(0,\pi)$.
	Then the number of eigenvalues counted with multiplicities of the operator $J$ lying in the interval $(-1/2,1/2)$ is at most four. 
\end{prop}
The above proposition is a special case of Proposition \ref{prop:def_bnd2}, which was formulated for the general star-graph with $N$ edges and therefore appears later in the manuscript.

\subsection{Example: equal interaction strengths on each edge}\label{ss-twoleads}

The assumption that the interaction strengths are related via $\tau_r=a\tau_l$, for some $a\in\RR$, can be made without loss of generality. Unfortunately, this level of generality still leaves it very difficult to solve equation \eqref{e-2eigenvalues} for any single variable and we are forced to restrict ourselves to specific values of $a$ which offer further simplifications. In this subsection, we assume that the interactions on both edges are equal, i.e.~ $\tau_l=\tau_r=\tau$ so that $a=1$, and try to determine how many eigenvalues of $J-1/2$ lie in $(-1,0)$ without considering the multiplicities. The analysis of multiplicities is subtle and therefore postponed to Subsection~\ref{ss-orthowedge}.

This scenario is also discussed in \cite{PVDB} for $C^2$ domains with a single straight-edged corner; see, in particular, Proposition 3.3 therein where a similar conclusion is reached. 
Subsection \ref{ss-twoleads2} will discuss the case where $a=-1$. Recall that we tacitly assume $\eta_l=\eta_r=0$ and $\tau\neq\pm2$.

After some rearranging, equation \eqref{e-2eigenvalues} becomes
\begin{align}\label{e-2equal}
(4-\tau^2)^2-(4+\tau^2)^2\cos(2\pi\la)-16
\tau^2\cos[2(\pi\la-2\w\la-\w)]=0.
\end{align}
The left-hand side of this equation can be viewed as function of $\la$ with other parameters fixed, and its roots analyzed to determine how many eigenvalues lie in the interval $(-1,0)$. 

\begin{prop}\label{p-2equal}
Given the above assumptions, for all $\tau\in\RR\setminus\{-2,0,2\}$ and $\w\in(0,\pi)$ such that $\w\neq\pi/2$ there exists one $\la_1\in(-1,-1/2)$ and one $\la_2\in(-1/2,0)$ that satisfy equation \eqref{e-2equal} while $\la=-1/2$ does not satisfy~\eqref{e-2equal}. For $\w=\pi/2$, there are no solutions of \eqref{e-2equal} in the interval $(-1,0)$.
\end{prop}

\begin{proof}
	First, consider the case $\w=\pi/2$. Equation~\eqref{e-2equal} reduces to $\cos(2\pi\la)=1$ and it is obvious that there are no solutions in the interval $(-1,0)$. Then, in light of Remark \ref{rem:symmetry}, it is sufficient to show the result for $\w\in(0,\pi/2)$.
	Consider the function
	\[
		F_{\tau,\omega}(\lambda) := (4-\tau^2)^2 - (4+\tau^2)^2\cos(2\pi\lambda)-16\tau^2\cos(2\pi\lambda-4\omega\lambda-2\omega),
	\]
	and notice the condition~\eqref{e-2equal} is equivalent to $F_{\tau,\omega}(\lambda) = 0$. It is also easy to see that 
	\[
		F_{\tau,\omega}(-1-\lambda)=F_{\tau,\omega}(\lambda),
	\]
	and hence $F_{\tau,\omega}(\lambda)$ is symmetric with respect to the point $\lambda = -1/2$. 
	Furthermore,
	\[
		F_{\tau,\omega}(-1) = 		F_{\tau,\omega}(0) = (4-\tau^2)^2-(4+\tau^2)^2-16\tau^2\cos(2\omega) = -16\tau^2(1+\cos(2\omega)) < 0.
	\]
	and
	\[
		F_{\tau,\omega}(-1/2) = (4-\tau^2)^2+(4+\tau^2)^2 + 16\tau^2 >0.
	\]
	This means the number of zeros of the function $F_{\tau,\omega}$ on the interval $(-1/2,0)$ is odd. But symmetry and Proposition~\ref{prop:def_bnd2} imply that $F_{\tau,\omega}$ has a maximum of two zeros in the interval $(-1/2,0)$. We conclude that $F_{\tau,\omega}$ has exactly one zero in $(-1/2,0)$.
\end{proof}

\subsection{Example: interaction strengths
	of opposite signs on the edges}\label{ss-twoleads2}

Another case that allows for simplification of equation \eqref{e-2eigenvalues} is when the singular interaction strength is the same on both edges but the normal vector from equation \eqref{e-polardirac} is flipped on one edge to point inwards towards $\Omega_+$. It can be calculated that this is equivalent to preserving the former orientation of the normal and negating the strength of the singular interaction on the chosen edge. A similar situation for Dirac operators with infinite-mass boundary conditions was analyzed in \cite{CL}. Without loss of generality, we choose this edge to be $\Gamma^r$ and set $-\tau_r =\tau_l = \tau$, as opposed to the previous example. Analogously to equation \eqref{e-2equal}, it is important to determine how many values $\la\in(-1,0)$ satisfy
\begin{align}\label{e-2opposite}
(4-\tau^2)^2-(4+\tau^2)^2\cos(2\pi\la)+16\tau^2\cos[2(\pi\la-2\w\la-\w)]=0.
\end{align}

\begin{prop}\label{p-2opposite}
Given the above assumptions, for all $\tau\in\RR\setminus\{-2,0,2\}$ and $\w\in(0,\pi)$ there exists one $\la_1\in(-1,-1/2)$ and one $\la_2\in(-1/2,0)$ that satisfy equation \eqref{e-2opposite},
while $\la =-1/2$ does not satisfy~\eqref{e-2opposite}.
\end{prop}

\begin{proof}
Assume without loss of generality, by Remark~\ref{rem:symmetry}, that $\w\in(0,\pi/2]$ and consider the function
	\[
		G_{\tau,\omega}(\lambda) := (4-\tau^2)^2 - (4+\tau^2)^2\cos(2\pi\lambda)+16\tau^2\cos(2\pi\lambda-4\omega\lambda-2\omega).
	\]
	The condition~\eqref{e-2opposite} is equivalent to $G_{\tau,\omega}(\lambda) = 0$. Notice that $G_{\tau,\omega}(-1-\lambda) = G_{\tau,\omega}(\lambda)$ and hence the function $G_{\tau,\omega}$ is symmetric with respect to the point $\lambda=-1/2$. 
	Furthermore,
	\[
	G_{\tau,\omega}(-1) = 		G_{\tau,\omega}(0) = (4-\tau^2)^2-(4+\tau^2)^2+16\tau^2\cos(2\omega) = -16\tau^2\big(1-\cos(2\omega)\big) < 0,
	\]
	and
	\[
	G_{\tau,\omega}(-1/2) = (4-\tau^2)^2+(4+\tau^2)^2 - 16\tau^2 =
	2(16-8\tau^2+
	\tau^4) = 2(4-\tau^2)^2 >0.
	\]
	The number of zeros of the function $G_{\tau,\omega}$ on the interval $(-1/2,0)$ is therefore odd. Again, symmetry and the bound in Proposition~\ref{prop:def_bnd2} imply that the number of zeros of $G_{\tau,\omega}$ on the interval $(-1/2,0)$ is no greater than two. The result follows.
\end{proof}

\subsection{Example: interaction supported on a ray}\label{ss-slit}

It is also possible to analyze the very degenerate case where the interaction strength on one edge is set to $0$, so that the situation collapses to just one ray. While the analysis of this non-closed curve with a Lorentz-scalar interaction is degenerate in the context of this manuscript and geometric setup, it has not been previously analyzed by any other method. Without loss of generality, we choose $\tau_r=0$ in order to have the notation as simple as possible. Equation \eqref{e-2eigenvalues} then reduces to 
\begin{align}\label{e-ray}
1-\frac{m_l\cos(2\pi\la)}{p_l}=0,
\end{align}
and allows for the number of eigenvalues in the appropriate intervals to be found. 

\begin{prop}\label{p-slit}
For $\tau_r = 0$, $\tau_l\in\RR\setminus\{-2,0,2\}$ and $\w\in(0,\pi)$, there exists one $\la_1\in(-1,-1/2)$ and one $\la_2\in(-1/2,0)$ that satisfy equation \eqref{e-ray}, while $\la=-1/2$ does not satisfy equation \eqref{e-ray}.
\end{prop}

\begin{proof}
Finding a family of solutions for equation \eqref{e-ray} is not difficult, and clearly does not depend on $\w$. For $\tau_l\neq-2,0,2$, the two solutions
\begin{align*}
\la_1&=-1+\arccos{\left(\frac{4-\tau_l^2}{4+\tau_l^2}\right)}/2\pi, \\
\la_2&=-\arccos{\left(\frac{4-\tau_l^2}{4+\tau_l^2}\right)}/2\pi,
\end{align*}
yield eigenvalues in the intervals $(-1,-1/2)$ and $(-1/2,0)$, respectively. Note that the value $\la=-1/2$ can be achieved by setting $\tau_l=0$, but the situation then corresponds to the free Dirac operator in $\RR^2$.
\end{proof}

\subsection{Orthogonal decomposition for the broken line}\label{ss-orthowedge}

In this subsection, we decompose the Dirac operator $\fD$ with Lorentz-scalar $\delta$-shell interactions supported on a broken line into the orthogonal sum of half-line Dirac operators with off-diagonal Coulomb-like potentials. This orthogonal decomposition then allows for the deficiency indices to be computed in terms of the spectrum of the spin-orbit operator $J$. For equal interaction strengths on the edges and interaction strengths of opposite signs on the edges, as in Subsections \ref{ss-twoleads} and \ref{ss-twoleads2} respectively, the deficiency indices are given explicitly. Computation of the deficiency indices crucially requires knowledge of how many eigenvalues of the operator $J$, counting multiplicities, lie in the interval $(-1/2,1/2)$.

The normalized eigenfunctions for $J$ form an orthonormal basis for the space $L^2(\SS^1,\CC^2) = L^2(\II_+,\CC^2)\oplus L^2(\II_-,\CC^2)$ and Corollary~\ref{c-2specsymmetric} says that the spectrum of $J$ is symmetric about the origin. Eigenspaces of $J$ will be denoted by
\begin{align*}
    \cF_{\widetilde{\la}}:=\ker(J-\widetilde{\la}),
\end{align*}
where $\cF_{\widetilde{\la}} = \{0\}$ if $\widetilde{\la}\notin\sigma(J)$. The dimension of the eigenspace is denoted by
\[
	n_{\widetilde\la} := \dim\cF_{\widetilde{\la}}.
\]
An orthogonal decomposition will require knowledge of how the subspace $\cF_{\widetilde{\la}}$ behaves when the unitary operator $\cS\colon L^2(\mathbb{S}^1,\CC^2)\rightarrow L^2(\mathbb{S}^1,\CC^2)$, $\cS\varphi:=(\sigma\cdot\va{e}\ti{rad})\varphi$ is applied, due to the polar decomposition in equation \eqref{e-polardirac}. The statement of the following proposition is necessarily slightly stronger than other similar results in the literature only because the eigenfunctions are not explicitly determined. See, e.g.~\cite[Theorem 3.2]{PVDB} where Lorentz-scalar interactions on $C^2$ domains with finitely many corners are considered, for a comparison.

\begin{prop}\label{p-2bijection}
The unitary operator $\cS$ is a bijection between the eigenspaces $\cF_{\widetilde{\la}}$ and $\cF_{-\widetilde{\la}}$ for all $\widetilde{\la}\in\RR$. 
\end{prop}

\begin{proof}
Set $\la := \widetilde{\la}-1/2$.
Let $\varphi\in\cF_{\widetilde{\la}}$ be arbitrary. It can be represented as
\[
\varphi(\te) := \begin{cases}
(
c_{+,1}e^{i\la\te},
c_{+,2}e^{-i\la\te})^\top
,\qquad\te\in(-\w,\w),\\
(c_{-,1}e^{i\la\te},
c_{-,2}e^{-i\la\te})^\top
,\qquad\te\in(\w,2\pi-\w),\\
\end{cases}
\]
with certain constants $c_{\pm,j}\in\CC$, $j=1,2$.
Begin by defining
\begin{align*}
    \Phi:=\cS\f=(\sigma\cdot\va{e}\ti{rad})\f.
\end{align*}
In order to prove that $\cS$ maps $\cF_{\widetilde{\la}}$ onto $\cF_{-\widetilde{\la}}$ we start by showing $\Phi\in\cF_{-\widetilde{\la}}$. Calculate that 
\begin{align*}
    \Phi&=
    \begin{pmatrix}
    0 & e^{-i\te} \\
    e^{i\te} & 0
    \end{pmatrix}\f
    =
    \begin{cases}
    (
    c_{+,2}e^{i\te(-\la-1)},
    c_{+,1}e^{-i\te(-\la-1)})^\top
    ,\qquad\te\in(-\w,\w),\\
    (c_{-,2}e^{i\te(-\la-1)},
    c_{-,1}e^{-i\te(-\la-1)})^\top
    ,\qquad\te\in(\w,2\pi-\w).\\
    \end{cases}
\end{align*}
It is clear that $(-i\sigma_3\Phi'+ \frac12\Phi) = -\widetilde{\la}\Phi$ on the intervals $(-\w,\w)$ and $(\w,2\pi-\w)$, but remains to show that $\Phi$ satisfies the boundary conditions necessary to be in $\dom J$. However, $\f\in\dom J$ already. Conditions on the constants can be extracted from the matrices $M_l$ and $M_r$, see equation \eqref{e-angular}, and are given in equation \eqref{e-2matrixequations}. Likewise, a similar calculation reveals that $\Phi$ satisfies the boundary conditions given in equation \eqref{e-angular} exactly when the four conditions in equation \eqref{e-2matrixequations} are satisfied. Therefore, $\Phi\in\cF_{-\widetilde{\la}}$ and $\Ran(\cS|_{\cF_{\widetilde\la}})\subset
	\cF_{-\widetilde{\la}}$.

The previous discussion immediately implies that $\cS$ is injective: if $\cS(\f)=\cS(\psi)$ for $\f,\psi\in\cF_{\widetilde\la}$ it is clear that $\f =\psi$ because their constants must be identical. Alternatively, one could notice that $\cS^2=I$. 

Finally, to show $\cF_{-\widetilde{\la}}\subset\Ran(\cS|_{
\cF_{\widetilde{\la}}})$, let $\widetilde{\f}\in\cF_{-\widetilde{\la}}$ be arbitrary and determined by constants denoted $\widetilde{c}_{\pm,1}$ and $\widetilde{c}_{\pm,2}$. Then $\widetilde{\f}$ is the image under $\cS$ of 
\begin{align*}
\cS^{-1}\widetilde{\f}=
\begin{cases}
(
\widetilde{c}_{+,2}e^{i\la\te},
\widetilde{c}_{+,1}e^{-i\la\te})^\top
,\qquad\te\in(-\w,\w),\\
(\widetilde{c}_{-,2}e^{i\la\te},
\widetilde{c}_{-,1}e^{-i\la\te})^\top
,\qquad\te\in(\w,2\pi-\w).\\
\end{cases}
\end{align*}
It is also clear that $\cS^{-1}\widetilde{\f}$ satisfies the boundary conditions for $\dom J$ interpreted via equation \eqref{e-2matrixequations} because $\widetilde{\f}$ does. Hence, $\cS^{-1}\widetilde{\f}\in\cF_{\widetilde{\la}}$ and the result follows.
\end{proof}

Of course, the operator $\cS$ being bijective between $\cF_{\wt\la}$ and $\cF_{-\wt\la}$ means that $n_{\widetilde{\la}} = n_{-\widetilde{\la}}$. However, equation \eqref{e-2matrixequations} already gives us insight into what values $n_{\widetilde{\la}}$ is allowed to take. 

\begin{prop}\label{c-simple}
Let $\tau_l,\tau_r\in\RR\setminus\{-2,0,2\}$
and $\w\in(0,\pi)$. Then the following statements hold.
\begin{itemize}
\item [{\rm (i)}]	
$\dim\cF_{\widetilde{\la}}=n_{\widetilde{\la}}\in\{1, 2\}$ for all $\widetilde{\la}\in\sigma(J)$. 
\item [{\rm (ii)}]If neither $\tau_l=\pm\tau_r$ nor $\tau_l\tau_r=\pm4$ then the spectrum of $J$ is simple.
\item [{\rm (iii)}]
If $\tau_l\tau_r\neq-4$ then the spectrum of $J$ in the interval $(-1/2,1/2)$ is simple and $0\notin\sigma(J)$. If $\tau_l\tau_r=-4$ then $J$ has a double eigenvalue at $0$. 
\end{itemize}
\end{prop}

\begin{proof}
Let $\f$ be an eigenfunction of $J$ corresponding to the eigenvalue $\widetilde\la$ with prescribed constants $c_{+,1}$ and $c_{+,2}$. Then the constants $c_{-,1}$ and $c_{-,2}$ are uniquely determined via the first two conditions in equation \eqref{e-2matrixequations} which come from the boundary conditions on $\Gamma^l$. Hence, $n_{\widetilde\la}\leq 2$ and item (i) is shown.

These constants $c_{-,1}$ and $c_{-,2}$ in turn prescribe an additional two conditions on $c_{+,1}$ and $c_{+,2}$ through the boundary conditions on $\Gamma^r$, written as the last two conditions in equation \eqref{e-2matrixequations}. The four conditions in equation \eqref{e-2matrixequations} can then be equated so that only $c_{+,1}$ and $c_{+,2}$ are unknown. The resulting conditions are:
\begin{equation}\label{eq:cond}
\begin{aligned}
    &c_{+,1}\left(\frac{m_l}{p_l}-\frac{m_re^{-i2\pi\la}}{p_r}\right)=-c_{+,2}\left(\frac{\tau_re^{i[\w(2\la+1)-2\pi\la]}}{p_r}+\frac{\tau_le^{-i\w(2\la+1)}}{p_l}\right), \\
    &c_{+,1}\left(\frac{\tau_le^{i\w(2\la+1)}}{p_l}+\frac{\tau_re^{-i[\w(2\la+1)-2\pi\la]}}{p_r}\right)=-c_{+,2}\left(\frac{m_l}{p_l}-\frac{m_re^{i2\pi\la}}{p_r}\right).
\end{aligned}
\end{equation}
Of course, if this system is written as a matrix then setting the determinant equal to $0$ recovers the general formula \eqref{e-2eigenvalues}. Eigenvalues of multiplicity two occur when all four entries of the matrix are zero. It is easy to see that two of the entries are real if and only if $e^{\pm i2\pi\la}\in\RR$. Hence, double eigenvalues require $\la=k$ or $\la=k+1/2$, where $k\in\ZZ$. 

Start by letting $\la=k$ for $k\in\ZZ$. Then $\la$ can be a double eigenvalue only if $\frac{m_l}{p_l}-\frac{m_r}{p_r}=0$, which happens only when $\tau_l=\pm\tau_r$. Plugging $\tau_l = \tau_r$ and $\la = k$ into the other entries from the system in equation \eqref{eq:cond} then implies that $\la= k$ is a double eigenvalue if
\begin{align}\label{e-doubleequaltaus}
    \frac{\tau_l\cos[\w(2k+1)]}{p_l}=0.
\end{align}
Equation \eqref{e-doubleequaltaus} is satisfied for 
\begin{align*}
    \w=\frac{\pi(2k+2s+1)}{2(2k+1)},\qquad s\in\ZZ.
\end{align*}
Plugging now $\tau_l = -\tau_r$ and $\la = k$ into the other entries from the system in equation \eqref{eq:cond} then implies that $\la = k$ is a double eigenvalue if
\begin{align}\label{e-doubleequaltaus2}
\frac{\tau_l\sin[\w(2k+1)]}{p_l}=0.
\end{align}
Equation \eqref{e-doubleequaltaus2} is satisfied for 
\begin{align*}
\w=\frac{\pi s}{2k+1},\qquad s\in\ZZ.
\end{align*}

Now consider the case $\la=k+1/2$ for $k\in\ZZ$. It is then necessary for $\frac{m_l}{p_l}+\frac{m_r}{p_r}=0$ and hence $\tau_l\tau_r=\pm4$. Item ${\rm (ii)}$ is thus proven. Plugging $\tau_l=4/\tau_r$ into the other entries from the system in equation \eqref{eq:cond} implies that $\la = k+1/2$ is a double eigenvalue if
\begin{align}\label{e-doublerequire1}
    \frac{\tau_r\cos[\w(2k+2)]}{\tau_r^2-4}=0,
\end{align}
must hold. This is true for $k\in\ZZ\backslash\{-1\}$
\begin{align*}
\w=\frac{\pi+2\pi s}{4k+4},\qquad s\in\ZZ
\end{align*}
Letting $\tau_l=-4/\tau_r$ yields a condition identical to equation \eqref{e-doublerequire1} except cosine is replaced by sine. For $k\in\ZZ$, we have a double eigenvalue if either $k =-1$, or $k\ne -1$ and 
\begin{align*}
\w=\frac{\pi s}{2(k+1)},\qquad s\in\ZZ.
\end{align*}
otherwise. All cases where eigenvalues $\la$ have multiplicity two have thus been analyzed. 

Critically, we notice that this last scenario, where $\tau_l\tau_r=-4$, is the only way a double eigenvalue can fall in the interval $(-1,0)$, and is due to the solution with $k=-1$ corresponding to $\la=-1/2$. Indeed, plugging $\la = -1/2$ into~\eqref{e-2eigenvalues} yields the condition
\[
	0 = (4-\tau_l^2)(4-\tau_r^2)
	+(4+\tau_l^2)(4+\tau_l^2)+ 16\tau_l\tau_r
	= 32 + 16\tau_l\tau_r + \tau_l^2\tau_r^2
	= (4+\tau_l\tau_r)^2,
\]
and so $\tau_l\tau_r=-4$ is a necessary condition for $\la=-1/2$ to be an eigenvalue. Item ${\rm (iii)}$ has been proven.
\end{proof}

\begin{rem}\label{rem:basis_for_zero}
Proposition~\ref{c-simple}\,(iii) says that $\tau_l\tau_r = -4$ is necessary and sufficient for $0\in\sigma(J)$. For $\tau_l\tau_r = -4$, we will show that it is possible to find an orthonormal basis $\{\f_0^1,\f_0^2\}$ of $\ker J$ such that $\f_0^2 = \cS\f_0^1$. Recall that by Proposition~\ref{p-2bijection}, for any normalized $\f\in\ker J$ one has $\cS\f\in\ker J$. It may happen that $\cS\f=\pm\f$, in which case $\f$ and $\cS\f$ are not linearly independent.

Let $\f_0^1$ and $\f_0^2$ be elements of $\ker(J)$ that are linearly independent. Let us show that both $\f_{0}^1$ and $\f_0^2$ cannot be such that $\cS\f_0^k = \f_0^k$, $k=1,2$. Suppose that $\cS\f_0^1=\f_0^1$. If the constants describing $\f_0^1$ are denoted $c_{\pm,j}$, $j=1,2$, then this means that $c_{+,1} = c_{+,2}$
and $c_{-,1} = c_{-,2}$. It is possible to write $c_{-,1}$ as a function of $c_{+,1}$ via the boundary conditions \eqref{eq:cond}, meaning that $\f_0^1$ is completely determined by $c_{+,1}$. Now assume that $\cS\f_0^2=\f_0^2$ too. Denoting the constants of $\f_0^2$ as $d_{\pm,j}$, $j=1,2$, the same line of reasoning means that $d_{+,1} = d_{+,2}$, $d_{-,1}=d_{-,2}$ and that $\f_0^2$ is determined by only $d_{+,1}$. This contradicts the statement that $\f_0^1$ and $\f_0^2$ are linearly independent. A similar argument shows that both $\f_0^1$ and $\f_0^2$ cannot be such that $\cS\f_0^k=-\f_0^k$, for $k=1,2$, with the main difference being that $c_{+,1} = -c_{+,2}$ and $c_{-,1} = -c_{-,2}$. As the restriction of $\cS$ to the $\ker(J)$ is acting on a two-dimensional space, after relabeling, we conclude that 
\begin{align*}
    \cS\cong\begin{pmatrix}
        1 & 0 \\
        0 & -1
    \end{pmatrix}.
\end{align*}
Now, given two normalized eigenfunctions of $\cS$, $\f_1$ and $\f_{-1}$, with eigenvalues $1$ and $-1$, respectively, it is possible to define
\begin{align*}
    \f_0^1 := \frac{1}{\sqrt{2}}(\f_1+\f_{-1})\quad\text{and}\quad
	\f_0^2 := \frac{1}{\sqrt{2}}(\f_1-\f_{-1}),
\end{align*}
so that the two functions are orthonormal, $\cS\f_0^1=\f_0^2$ and $\cS\f_0^2=\f_0^1$, as desired. 
\hfill$\diamondsuit$
\end{rem}


It is now possible to construct an orthogonal decomposition of the Hilbert space 
\begin{align*}
L^2_{\rm pol}(\RR^2,\CC^2) = L^2_{\rm pol}(\Omega_+,\CC^2)\oplus L^2_{\rm pol}(\Omega_-,\CC^2),
\end{align*}
in terms of the spectral decomposition of the spin-orbit operator $J$. Double eigenvalues of $J$ and the eigenvalue $0$ will require careful consideration.  

Let $\widetilde{\la}\in\sigma(J)\cap(0,\infty)$
with $n_{\widetilde\la} = 1$ and $\f_{\widetilde\la}^1\in\cF_{\widetilde\la}$ be normalized. Denote
\begin{align*}
    \cE_{\widetilde\la}^1 :=L^2_r(\RR_+)\otimes{\rm span}\{\f_{\widetilde\la}^1,\cS\f_{\widetilde\la}^1\}.
\end{align*}
Let $\widetilde{\la}\in\sigma(J)\cap(0,\infty)$
with $n_{\widetilde\la} = 2$ and
$\{\f^1_{\widetilde\la},\f^2_{\widetilde\la}\}$ be the orthonormal basis of
$\cF_{\widetilde\la}$. Denote
\begin{align*}
\cE_{\widetilde\la}^j :=L^2_r(\RR_+)\otimes{\rm span}\{\f_{\widetilde\la}^j,\cS\f_{\widetilde\la}^j\},\qquad j=1,2.
\end{align*}
When $0\in\sigma(J)$, let
$\{\f^1_{0},\f^2_{0}\}$ be the orthonormal basis of
$\cF_{0}$ constructed as in Remark~\ref{rem:basis_for_zero} and denote
\begin{align*}
\cE_{0} :=L^2_r(\RR_+)\otimes{\rm span}\{\f_{0}^1,\cS\f_{0}^1\}.
\end{align*}
The following orthogonal decomposition now holds if $0\notin\sigma(J)$
\begin{align*}
    L^2_{\rm pol}(\RR^2,\CC^2) \simeq L^2_r(\RR_+)\otimes\left[L^2(\II_+,\CC^2)\oplus L^2(\II_-,\CC^2)\right]
    =
    \bigoplus_{
    \widetilde\la\in\sigma(J)\cap(0,\infty)}\bigoplus_{j=1}^{n_{\widetilde\la}}\cE_{\widetilde\la}^j,
\end{align*}
while if $0\in\sigma(J)$ we have the slightly modified decomposition
\begin{align*}
L^2_{\rm pol}(\RR^2,\CC^2) \simeq \cE_0\oplus\left(
\bigoplus_{
	\widetilde\la\in\sigma(J)\cap(0,\infty)}\bigoplus_{j=1}^{n_{\widetilde\la}}\cE_{\widetilde\la}^j\right).
\end{align*}
Finally, it is convenient to introduce the
unitary transformations $W_{\wt\la}\colon\cE_{\wt\la}^j\to L^2(\RR_+,\CC^2)$, for $\wt\la\in\sigma(J)\cap(0,\infty)$ and $1\le j\le n_{\wt\la}$ that act via
\begin{align*}
    (W_{\widetilde\la}^j u)(r):=\sqrt{r}
    \begin{pmatrix}
    \langle u(r,\cdot),\f_{\widetilde\la}^j\rangle_{\SS^1} \\
    i\langle u(r,\cdot),\cS\f_{\widetilde\la}^j\rangle_{\SS^1}
    \end{pmatrix}.
\end{align*}
If $0\in\sigma(J)$, then the unitary mapping $W_0\colon\cE_0\rightarrow L^2(\RR_+,\CC^2)$ acts analogously via
\begin{align*}
(W_{0} u)(r):=\sqrt{r}
\begin{pmatrix}
\langle u(r,\cdot),\f_{0}^1\rangle_{\SS^1} \\
i\langle u(r,\cdot),\cS\f_{0}^1\rangle_{\SS^1}
\end{pmatrix}.
\end{align*}
The next theorem decomposes the operator $\fD$ into the orthogonal sum of half-line Dirac operators with off-diagonal Coulomb potentials. For $\wt\la\ge 0$, these half-line Dirac operators are symmetric and densely defined operators in the Hilbert space $L^2(\RR_+,\CC^2)$. They are denoted
\begin{equation}\label{eq:1dDirac}
\begin{aligned}
\bd_{\wt\la}\psi &:=
\begin{pmatrix}
0 & -\frac{d}{dr}-\frac{\widetilde{\la}}{r} \\
\frac{d}{dr}-\frac{\widetilde{\la}}{r} &  0
\end{pmatrix}\psi,\\
\dom\bd_{\wt\la}&:=
\begin{cases} H_0^1(\RR_+,\CC^2), &\quad\wt\la\ne 1/2,\\
\left\{\psi\colon 
\psi_1,\psi_1' - \frac{\psi_1}{2r}\in L^2(\RR_+),~ \psi_2\in H^1_0(\RR_+)\right\},&\quad\wt\la = 1/2.
\end{cases}
\end{aligned}
\end{equation} 
Appendix~\ref{app} shows that the operators $\bd_{\wt\la}$ are closed.

\begin{theo}\label{p-decomp}
For $\wt\la\in\sigma(J)\cap(0,\infty)$, the operators
\begin{align*}
d_{\wt\la}^j u:=\widetilde{\cD}u, \hspace{3em} \dom d_{\wt\la}^j:=\dom\,\widetilde{\fD}\cap\cE_{\wt\la}^j,~~~\text{ for }~1\le j\le n_{\wt\la},
\end{align*}
and, in the case that $0\in\sigma(J)$, the operator
\begin{align*}
d_0 u:=\widetilde{\cD}u, \hspace{3em} \dom d_0:=\dom\,\widetilde{\fD}\cap\cE_0,
\end{align*}
are well defined in the Hilbert spaces $\cE_{\wt\la}^j$ and $\cE_0$, respectively. The operator $d_{\wt\la}^j$ is unitarily equivalent via $W_{\wt\la}^j$
to the half-line Dirac operator $\bd_{\wt\la}$ in $L^2(\RR_+,\CC^2)$ defined in~\eqref{eq:1dDirac}, while the operator $d_0$ is unitarily equivalent via $W_0$ to $\bd_0$.
In particular, the decomposition
\begin{align*}
    \fD\simeq
    \begin{cases}\bigoplus_{\wt\la\in\sigma(J)\cap (0,\infty)}\bigoplus_{j=1}^{n_{\wt\la}}\bd_{\wt\la},& 0\notin\sigma(J),\\
    \bd_0\oplus
    \left(\bigoplus_{\wt\la\in\sigma(J)\cap (0,\infty)}\bigoplus_{j=1}^{n_{\wt\la}}\bd_{\wt\la}\right),& 0\in\sigma(J).
    \end{cases}
\end{align*}
holds, the operator $\fD$ is closed and its deficiency indices can be computed as 
\[
n_{\pm}(\fD)=
\begin{cases}
\sum_{\wt\la\in\sigma(J)\cap(0,\infty)}
n_{\wt\la} n_{\pm}(\bd_{\wt\la}),& 0\notin\sigma(J),\\
n_\pm(\bd_0)+\sum_{\wt\la\in\sigma(J)\cap(0,\infty)}
n_{\wt\la} n_{\pm}(\bd_{\wt\la}),& 0\in\sigma(J).
\end{cases}
\]
\end{theo}

\begin{proof}
We prove the theorem only for the case $0\notin\sigma(J)$, with an analogous argument holding when $0\in\sigma(J)$. Let $\wt\la\in\sigma(J)\cap (0,\infty)$.
A function $u\in\dom\widetilde{\fD}\cap\cE_{\wt\la}^j$, $1\le j\le n_{\wt\la}$, can be written as
\begin{align*}
    u(r,\te)=\frac{\psi_1(r)}{\sqrt{r}}\f_{\wt\la}^j(\te)-i\frac{\psi_2(r)}{\sqrt{r}}(\cS\f_{\wt\la}^j)(\te),
\end{align*}
for some $\psi_1,\psi_2\colon\RR_+\to\CC$.
It is easy to check that 
\begin{equation}\label{eq:fSf}
\begin{aligned}
	\left\langle(\f_{\wt\la}^j)',(\cS\f_{\wt\la}^j)'\right\rangle_{\SS^1} &= \left\langle(i\sigma_3\f_{\wt\la}^j)',(i\sigma_3\cS\f_{\wt\la}^j)'\right\rangle_{\SS^1} = 0,\\ \left\|(\cS\f^j_{\wt\la})'\right\|_{\SS^1}&\ne 0\quad\text{for all}\quad\wt\la \ge 0,\\
	\left\|(\f^j_{\wt\la})'\right\|_{\SS^1}&\ne 0\quad\text{if and only if}\quad \wt\la \ne 1/2.
\end{aligned}		
\end{equation}	 
Thanks to the orthogonality of $\f_{\wt\la}^j$ and $\cS\f_{\wt\la}^j$, equation \eqref{eq:fSf} and the expression for the gradient in polar coordinates we obtain that $u\in H^1_{\rm pol}(\Omega_+,\CC^2)\oplus H^1_{\rm pol}(\Omega_-,\CC^2)$ is equivalent to 
the conditions
$\psi_j,\psi_j' -\frac{\psi_j}{2r}, \frac{\psi_2}{r}\in L^2(\RR_+)$ for $j=1,2$ and, in addition, $\frac{\psi_1}{r}\in L^2(\RR_+)$ provided that $\wt\la \ne 1/2$.  Further, using estimates from \cite[Proposition 2.2/2.4]{CP}, we can obtain the equivalences
\begin{align*}
    u&\in H^1_{\rm pol}(\Omega_+,\CC^2)\oplus H^1_{\rm pol}(\Omega_-,\CC^2)\hspace{.5em}\iff\hspace{.5em}\psi_1,\psi_2\in H_0^1(\RR_+),\qquad\wt\la \ne 1/2,\\
    u&\in H^1_{\rm pol}(\Omega_+,\CC^2)\oplus H^1_{\rm pol}(\Omega_-,\CC^2)\hspace{.5em}\iff\hspace{.5em}\psi_1,\psi_1'-\frac{\psi_1}{2r}
    \in L^2(\RR_+),~\psi_2\in H_0^1(\RR_+),\quad\wt\la = 1/2.
\end{align*}
In order to prove that the operators $d_{\wt\la}^j$ are symmetric and that the orthogonal decomposition of $\fD$ holds, we must show that $\widetilde{\fD}(\dom\widetilde{\fD}\cap\cE_{\wt\la}^j)\subset\cE_{\wt\la}^j$ holds. Hence, calculate
\begin{align*}
    \widetilde{\fD}u&=-\frac{i(\sigma\cdot\va{e}\ti{rad})}{\sqrt{r}}\left[\f_{\wt\la}^j\left(\partial_r\psi_1-\frac{\wt\la\psi_1}{r}\right)-i\cS\f_{\wt\la}^j\left(\partial_r\psi_2+\frac{\wt\la\psi_2}{r}\right)\right] \\
    &=\frac{1}{\sqrt{r}}\left[-i\cS\f_{\wt\la}^j\left(\partial_r\psi_1-\frac{\widetilde{\la}\psi_1}{r}\right)+\f_{\wt\la}^j\left(-\partial_r\psi_2-\frac{\wt\la\psi_2}{r}\right)\right]. 
\end{align*}
It is then easy to see that by construction
\begin{align*}
    	W_{\wt\la}^j d_{\wt\la}^j (W_{\wt\la}^j)^{-1}=\bd_{\wt\la},
\end{align*}
holds. The result follows. 
\end{proof}

The deficiency indices of $\fD$ can now easily be identified.

\begin{prop}\label{prop:def_indices_bl}
	The deficiency indices $n_\pm(\fD)$ of the Dirac operator $\fD$ are both equal to half of the number of eigenvalues of the spin-orbit operator $J$ lying in the interval $(-1/2,1/2)$ with multiplicities taken into account. In particular, $n_+(\fD) = n_-(\fD) \le 2$.
\end{prop}
\begin{proof}
	The operator $\bd_{\wt\la}$ for $\wt\la\ge 1/2$ is self-adjoint by~\cite[Theorem 1.1]{CP}.
	While for $\wt\la\in [0,1/2)$ the operator $\bd_{\wt\la}$ has deficiency indices $(1,1)$. Here, the analysis for the case $0 < \wt\la < 1/2$ can be found
	in \cite[Theorem 1.2(i)]{CP} while the case $\wt\la = 0$ is standard. Hence, the characterization of the deficiency indices of $\fD$ follows from Theorem~\ref{p-decomp}. Combining this fact with Proposition~\ref{prop:def_bnd2} yields $n_+(\fD) = n_-(\fD) \le 2$.
\end{proof}
Proposition \ref{prop:def_indices_bl} can immediately be applied to the examples from Subsections~\ref{ss-twoleads} and~\ref{ss-twoleads2}, which deal with the special cases when $\tau_l=\tau_r$ and $-\tau_l=\tau_r$, respectively, because the number of eigenvalues in the interval $(-1/2,1/2)$ were already determined. 

\begin{cor}\label{c-2defindices}
Let $\w\in(0,\pi)$. Then the following hold.
\begin{itemize}
	\item [{\rm (i)}]
	If $\tau_l= \tau_r\notin \{-2,0,2\}$, then the deficiency indices of $\fD$ are $n_+(\fD) = n_-(\fD)= 1$ if $\w\ne \frac{\pi}{2}$, while if $\w = \frac{\pi}{2}$ then $\fD$ is  self-adjoint.
	\item [{\rm (ii)}]
	If $\tau_l= -\tau_r\notin \{-2,0,2\}$, then the deficiency indices of $\fD$ are $n_+(\fD) = n_-(\fD)= 1$.
	\item [{\rm (iii)}] If either $\tau_l=0$ or $\tau_r=0$ and the nonzero parameter is not equal to $-2$ or $2$, then the deficiency indices of $\fD$ are $n_+(\fD) = n_-(\fD)= 1$.
\end{itemize}	
\end{cor}
\begin{proof}
(i) For $\w \ne \frac{\pi}{2}$,  Proposition~\ref{p-2equal} yields that
the operator $J$ has two eigenvalues in the interval $(-1/2,1/2)$, one in each of the intervals $(-1/2,0)$ and $(0,1/2)$. According to Proposition~\ref{c-simple}\,(iii), these eigenvalues are simple. For $\w = \frac{\pi}{2}$,  Proposition~\ref{p-2equal} says that the operator $J$ has no eigenvalues in the interval $(-1/2,1/2)$.
Hence, the claim follows from Proposition~\ref{prop:def_indices_bl}.

(ii) Proposition~\ref{p-2opposite} states that the operator $J$ has two eigenvalues in the interval $(-1/2,1/2)$ and these are each simple by Proposition~\ref{c-simple}\,(iii). The claim then follows from Proposition~\ref{prop:def_indices_bl}.

(iii) Proposition~\ref{p-slit} says that $J$ has two eigenvalues in $(-1/2,1/2)$ that are again simple. As $\tau_l$ and $\tau_r$ are interchangeable in Proposition \ref{p-slit}, the claim follows. 
\end{proof}

\section{Singular interactions supported on a star graph}\label{s-star}

We can now consider a generalization of the former model to the singular interaction supported on a star-graph in $\RR^2$. Again, we are able to reduce the computation of the deficiency indices to the spectral analysis of the respective spin-orbit operator by using an orthogonal decomposition. This allows for a careful analysis to be conducted when the star-graph has $N=3$ edges in Subsection \ref{ss-threeleads}. If, in addition, this star-graph with three edges is symmetric, deficiency indices can be computed concretely, see Subsection \ref{ss-threeequalleads}. For symmetric star-graphs in general, an alternative spectral condition is introduced in \ref{ss-altspec} which exploits the connection the Dirac operators have with momentum operators on a certain graph. The spectral analysis of the spin-orbit operator then simplifies significantly and allows us to present in Subsection \ref{ss-22} a configuration of interaction strengths for which the deficiency indices are higher than $(1,1)$.
 
Consider the star-graph with $N$ edges meeting at the origin with $N\ge2$, each with potentially different singular interaction strengths. It is first necessary to alter the previously used notation a bit to denote more sectors. Adjusting from Section~\ref{s-wedge}, denote $\Gamma^r:=\Gamma_1$, $\Gamma^l:=\Gamma_2$ and $\w:=\w_1$. Begin by fixing a sequence of angles $\w_j$, for $j\in\{1,\dots,N-1\}$, whose values are strictly increasing as $j$ increases and such that $\w_{N-1}<2\pi-\w_1$. The $N$ edges are then defined by
\begin{align*}
    \Gamma_1&:=\left\{(r\cos\w_1,-r\sin\w_1)\in\RR^2~:~r>0\right\}, \text{ and } \\
    \Gamma_j&:=\{(r\cos\w_{j-1},r\sin\w_{j-1})\in\RR^2~:~r>0\},
\end{align*}
for $j\in\{2,\dots,N\}$.
The respective star-graph is defined by $\Gamma = \cup_{j=1}^N \Gamma_j$.
Intervals for the angular coordinates are given by
\begin{align*}
    \II_j:=(\w_{j-1},\w_j) \text{ for }j\in\{1,\dots,N\},
\end{align*}
with the convention that $\w_0=-\w_1$ and $\w_{N}=2\pi-\w_1$.
The sectors are likewise described via
\begin{align*}
    \Omega_j:=\{(r\cos\te,r\sin\te)\in\RR^2~:~r>0,~\te\in\II_j\}\subset\RR^2,
\end{align*}
for $j\in\{1,\dots,N\}$. Normal vectors for each $\Gamma_j$, denoted $\nu_j$, are assumed to be unit vectors and have fixed clockwise orientation. 
Spinors $u\in L^2(\RR^2,\CC^2)$ are decomposed into $N$ parts corresponding to the $N$ sectors as
\begin{align*}
u=\bigoplus_{j=1}^N u_j\in L^2(\RR^2,\CC^2) =  \bigoplus_{j=1}^N L^2(\Omega_j,\CC^2),
\end{align*}
where $u_j := u|_{\Omega_j}$.
The Dirac operator with the Lorentz-scalar $\delta$-shell interaction supported on $\Gamma$ acts in the Hilbert space $L^2(\RR^2,\CC^2)$  as
\begin{equation}\label{Dirac_N}
\begin{aligned}
\fD_N u&:= \bigoplus_{j=1}^N\cD u_j , \\
\dom\fD_N&:=\Bigg\{u=\bigoplus_{j=1}^Nu_j\in \bigoplus_{j=1}^N H^1(\Omega_j,\CC^2)~:~\forall j\in\{1,\dots,N\}\\
&-i(\sigma_1\nu_{j,1}+\sigma_2\nu_{j,2})\left(u_{j-1}|_{\Gamma_j}-u_j|_{\Gamma_j}\right)=\frac{1}{2}\tau_j\sigma_3\left(u_{j-1}|_{\Gamma_j}+u_j|_{\Gamma_j}\right)\Bigg\},
\end{aligned}
\end{equation}
using the convention that $u_0=u_N$ to simplify notation. We will denote by $\wt{\fD}_N$ the counterpart of this operator in polar coordinates. In order to make the presentation less technical, only Lorentz-scalar $\delta$-shell interactions are considered, not electrostatic $\delta$-shell interactions. However, it is possible to proceed in the same way as Section \ref{s-wedge} and determine that the spin-orbit operator $J_N$, given below in equation \eqref{e-angular3}, is not self-adjoint unless there are no electrostatic interactions. 

\begin{rem}
The clockwise choice of orientation of the normal vectors is made without loss of generality, as choosing an interaction strength for a coupling parameter $\tau_j$ on the lead $\Gamma_j$ with counterclockwise normal vector is equivalent to choosing $-\tau_j$ as a coupling parameter on $\Gamma_j$ with a clockwise normal vector. 
Hence, to compare results in this section for $N =2$ with Section \ref{s-wedge} where the orientation of normal vectors is different it is necessary to start with $\tau_l = -\tau_2$ and $\tau_r=\tau_1$ in order to make the boundary conditions in equation \eqref{Dirac_N} for $N =2$
equivalent to the boundary conditions
in equation \eqref{e-dirac} for $\eta_l = \eta_r =0$.
\hfill$\diamondsuit$
\end{rem}

The abbreviations $m_j$, $p_j$ and $\eps_j$ will be used in analogy to equation \eqref{e-pm}. These boundary conditions can still be rewritten if we assume that the non-confining case holds on each edge, i.e.~$\tau_j\neq\pm2$ for each $j\in\{1,\dots,N\}$. Explicitly, for $j\in\{1,\dots,N\}$ and using the conventions that $u_0=u_N$ and $\w_0=-\w_1$, we have
\begin{align}\label{e-betterbcsN}
u_{j-1}|_{\Gamma_j}&=\frac{1}{p_j}
\begin{pmatrix}
m_j & e^{-i\w_{j-1}}\tau_j \\
e^{i\w_{j-1}}\tau_j & m_j
\end{pmatrix}
u_j|_{\Gamma_j}.
\end{align}
The matrix in the boundary condition concerning $\Gamma_j$ will be referred to as $M_j$. The operator $\fD_N$ can easily be shown to be symmetric and densely defined in analogy with Proposition \ref{p-symmetric} because these matrices of boundary conditions still adhere to the relations given in equation \eqref{e-adjointsrelation}.

The associated spin-orbit operator acts in the Hilbert space $L^2(\SS^1,\CC^2) \!=\! \bigoplus_{j=1}^N L^2(\II_j,\CC^2)$ as 
\begin{equation}\label{e-angular3}
\begin{aligned}
J_N\f&:=\bigoplus_{j=1}^N
\left(-i\sigma_3\f_j'+\frac12\f_j\right),\\
\dom J_N&=\Bigg\{\f=\bigoplus_{j=1}^N\f_j\in \bigoplus_{j=1}^N H^1(\II_j,\CC^2)\colon
\f_N(2\pi-\w_1) = M_1\f_1(-\w_1),\\
&\qquad\qquad\qquad\qquad\qquad\qquad\f_{j-1}(\w_{j-1})=M_j
\f_j(\w_{j-1}), \forall j\in\{2,\dots,N\} \Bigg\}.
\end{aligned}
\end{equation}

\begin{prop}\label{p-sathree}
Assume that $\tau_j\neq\pm2$ for all $j\in\{1,\dots,N\}$. Then the operator $J_N$ is self-adjoint and has compact resolvent.
\end{prop}

\begin{proof}
The operator $J_N-1/2$ can be viewed as a momentum operator on a graph, as in Proposition \ref{p-momentum2}, with $2N$ directed edges, vectors on them denoted by $\psi_j$ for $j\in\{1,\dots,2N$\}, and $N$ vertices: $-\w_1,\w_1,\dots,\w_{N-1}$. If we denote $\w_0=\w_N=-\w_1$, then edges with odd indices are chosen to be oriented from $\w_{j-1}$ to $\w_j$. Edges with even indices are oriented in the reverse direction: going from $\w_j$ to $\w_{j-1}$. In this way, $\psi_{2k-1}$ and $\psi_{2k}$, for $k\in\{1,\dots,N\}$, represent the two components of $\f_k$. The Hamiltonian acting on these vectors is thus $-i\frac{d}{dx}$ and the effect of $\sigma_3$ in $J_N-1/2$ is accounted for by the direction of the edges. Define the vectors
\begin{align*}
    \psi\ti{out}&:=(\psi_1(-\w_1),\psi_2(\w_1),\psi_3(\w_1),\psi_4(\w_2),\psi_5(\w_2),\psi_6(\w_3),\dots,\psi_{2N-1}(\w_{N-1}),\psi_{2N}(2\pi-\w_1)) \\
    \psi\ti{in}&:=(\psi_1(\w_1),\psi_2(-\w_1),\psi_3(\w_2),\psi_4(\w_1),\psi_5(\w_3),\psi_6(\w_2),\dots,\psi_{2N-1}(2\pi-\w_1),\psi_{2N}(\w_{N-1})),
\end{align*}
which list the terminals and origins of the vectors $\psi_j$, respectively. Note that the matrix $M_1$ is identical for the angles $-\w_1$ and $2\pi-\w_1$, which allows for the vertices $-\w_1$ and $2\pi-\w_1$ to be identified and simply denoted by $-\w_1$.

Self-adjoint realizations of this Hamiltonian are in one-to-one correspondence with $2N\times 2N$ unitary matrices $\cU$ which realize $\cU\psi\ti{in}=\psi\ti{out}$ by \cite[Proposition 4.1]{E}. The matrix $\cU_N$
corresponding to $J_N - 1/2$ relies completely on the matrices $M_j$ from equation \eqref{e-betterbcsN} now.  

It is clear that each individual $M_j$ relates two coordinates of $\psi\ti{out}$ to two coordinates of $\psi\ti{in}$; for $M_j$ the outgoing coordinates are $\psi_{2j-2}(\w_{j-1})$ and $\psi_{2j-1}(\w_{j-1})$, while the incoming coordinates are $\psi_{2j-3}(\w_{j-1})$ and $\psi_{2j}(\w_{j-1})$, for $j\in\{1,\dots,N\}$ with $\w_0=-\w_1$. An index of $-k$ is considered as $2N-k$ while an index of $0$ is considered as $2N$ here. Hence, omitting the dependence on $\w_{j-1}$ as this is where all functions are evaluated, we obtain
\begin{align*}
    \psi_{2j-2}&=\frac{e^{i\w_{j-1}}\tau_j}{m_j}\psi_{2j-3}+\frac{p_j}{m_j}\psi_{2j} \\
    \psi_{2j-1}&=\frac{p_j}{m_j}\psi_{2j-3}-\frac{e^{-i\w_{j-1}}\tau_j}{m_j}\psi_{2j}.
\end{align*}
Each $\w_{j-1}$ with $j\in\{1,\dots,N\}$ will thus generate four entries for $\cU_N$:
\begin{equation}\label{e-Uentries}
\begin{aligned}
    (\cU_N)_{(2j-2,2j-3)}&=\frac{e^{i\w_{j-1}}\tau_j}{m_j},\qquad &(\cU_N)_{(2j-2,2j)}&=\frac{p_j}{m_j}, \\
    (\cU_N)_{(2j-1,2j-3)}&=\frac{p_j}{m_j},
    \qquad &(\cU_N)_{(2j-1,2j)}&=-\frac{e^{-i\w_{j-1}}\tau_j}{m_j}.
\end{aligned}
\end{equation}
A direct computation reveals the only non-zero entries of the matrix $\cU_N\cU_N^*$ are on its diagonal and  
\begin{align*}
    (\cU_N\cU_N^*)_{2j-1,2j-1} = (\cU_N\cU_N^*)_{2j-2,2j-2}=\left[\frac{\tau_j}{m_j}\right]^2+\left[\frac{p_j}{m_j}\right]^2= 1,\quad \forall j\in \{1,\dots, N\}. 
\end{align*}
The matrix $\cU_N$ is thus unitary. A similar argument holds for $\cU_N^*\cU_N$, and thus the operator $J_N$ is self-adjoint. Compactness of the embedding of $\bigoplus_{j=1}^N H^1(\II_j,\CC^2)$ into $L^2(\SS^1,\CC^2)$ reveals that the resolvent of $J_N$ is also compact.
\end{proof}

Again, normalized eigenfunctions of the operator $J_N$ are guaranteed to form an orthonormal basis of $L^2(\SS^1,\CC^2)$. Let $\f=\bigoplus_{j=1}^N(\f_{j,1},\f_{j,2})$ and $\widetilde{\la}\in\RR$ satisfy $J_N\f=\widetilde{\la}\f$. Then, subtracting $1/2$ from both sides, $\f$ satisfies
\begin{align*}
    &-i\f_{j,1}'=(\widetilde{\la}-1/2)\f_{j,1}, \\
    &+i\f_{j,2}'=(\widetilde{\la}-1/2)\f_{j,2},
\end{align*}
for each $j\in\{1,\dots,N\}$. As in Section \ref{s-wedge}, setting $\la:=\widetilde{\la}-1/2$ will be convenient. The generic solution for this system of equations is the same as before, shown in equation \eqref{e-genericeigenfunction}, but there are now $2N$ constants: $c_{j,1},c_{j,2}\in\CC$ for $j\in\{1,\dots,N\}$. These generic eigenfunctions can be plugged into the boundary conditions from equation \eqref{e-angular3} with the relevant angles. For $j\in\{2,\dots,N\}$, this yields that the conditions for $\f$ to be an eigenfunction are:
\begin{equation}\label{e-Nmatrixequations}
\begin{aligned}
    c_{j-1,1}&=c_{j,1}\frac{m_j}{p_j}+c_{j,2}\frac{\tau_je^{-i\w_{j-1}(2\la+1)}}{p_j}, \\
    c_{j-1,2}&=c_{j,1}\frac{\tau_je^{i\w_{j-1}(2\la+1)}}{p_j}+c_{j,2}\frac{m_j}{p_j}.
\end{aligned}
\end{equation}
Additional conditions stemming from the matrix $M_1$ are a bit different, as it relates eigenfunctions with angles $2\pi-\w_1$ and $-\w_1$, respectively. The conditions are
\begin{equation}\label{e-Nmatrixequations2}
\begin{aligned}
    c_{N,1}&=c_{1,1}\frac{m_1e^{-i2\pi\la}}{p_1}+c_{1,2}\frac{\tau_1e^{i[\w_1(2\la+1)-2\pi\la]}}{p_1}, \\
    c_{N,2}&=c_{1,1}\frac{\tau_1e^{-i[\w_1(2\la+1)-2\pi\la]}}{p_1}+c_{1,2}\frac{m_1e^{i2\pi\la}}{p_1}.
\end{aligned}
\end{equation}

The described system of equations will have a non-trivial solution if the determinant of the corresponding $2N\times 2N$ matrix vanishes. This matrix, say $\cM_N$, has three non-zero entries in each row, given above. Unfortunately, there seems to be no way to explicitly state the determinant of $\cM_N$ for general $N$, but the case $N=3$ is possible and will be analyzed in Subsection \ref{ss-threeleads}.

However, separation of variables can still continue along the lines of Subsection \ref{ss-orthowedge}. To this end, for $\wt\la\in\RR$, define the spaces
\begin{align*}
    \cF_{\wt\la}^N :=\ker(J_N-\wt\la).
\end{align*}
As in Proposition \ref{p-2bijection}, it is possible to analyze the effect of the operator
$\cS\colon L^2(\SS^1,\CC^2)\rightarrow L^2(\SS^1,\CC^2)$ acting as $\cS\f=(\sigma\cdot\va{e}\ti{rad})\f$  on the functions in $\cF_{\wt\la}^N$. The proof of the following proposition is analogous
to that of Proposition \ref{p-2bijection}.

\begin{prop}\label{p-sbijection}
	The operator $\cS$ is a bijection between $\cF_{\wt\la}^N$ and $\cF_{-\wt\la}^N$ for all $\wt\la\in\RR$.
\end{prop}

Of course, the operator $\cS$ being bijective means that $\dim\cF_{\wt\la}^N=\dim\cF_{-\wt\la}^N:=n_{\wt\la}$ and $\cF_{\wt\la}^N$ is non-trivial only for $\widetilde{\la}\in\sigma(J_N)$. Hence, the spectrum of $J_N$ must be symmetric about the origin. The multiplicities of $J_N$ can not be controlled as precisely as in the case of the interaction supported on a broken line.

\begin{prop}
	Let $\{\tau_j\}_{j=1}^N$ and $\{\w_j\}_{j=1}^{N-1}$ be as above. Then the following statements hold.
	\begin{itemize}
		\item [{\rm (i)}]	
		$\dim\cF_{\widetilde{\la}}^N=n_{\widetilde{\la}}\in\{1, 2\}$ for all $\widetilde{\la}\in\sigma(J_N)$. 
		\item [{\rm(ii)}] If $0\in\sigma(J_N)$, then $n_0 = 2$.
	\end{itemize}
\end{prop}

\begin{proof}
Let $\f$ be an eigenfunction of $J_N$ corresponding to the eigenvalue $\widetilde\la$ with $c_{1,1}$ and $c_{1,2}$ prescribed. Then, the constants $c_{2,1}$ and $c_{2,2}$ are uniquely determined from the boundary condition on $\Gamma_2$, $c_{3,1}$ and $c_{3,2}$ are then determined from the boundary condition on $\Gamma_3$. Repeating this process, eventually all constants $c_{j,1}$ and $c_{j,2}$ describing the function $\f$ are given; leaving only two degrees of freedom for the function $\f$. Hence, the dimension of $\cF_{\wt\la}^N$ is at most two and item $(i)$ is shown.

If $0\in\sigma(J_N)$, then it is necessary to take a closer look at how these constants are determined for $\f\in\ker(J_N)$. Explicitly, multiplying $(c_{j,1},c_{j,2})^T$, for $j=1,\dots,N-1$, by the matrix \begin{align}\label{e-matrixaround}
    \fM_j:=\begin{pmatrix}
    \frac{m_j}{p_j} & -\frac{\tau_j}{p_j} \\
    -\frac{\tau_j}{p_j} & \frac{m_j}{p_j}
    \end{pmatrix},
\end{align}
yields the constants $c_{j+1,1}$ and $c_{j+1,2}$. The matrix $\fM_N$ has the same form as equation \eqref{e-matrixaround} except is multiplied by a factor of $-1$. Define $\fM:=\prod_{j=1}^N\fM_j$ and notice that $\det(\fM)=1$ because $\det(\fM_j)=1$ for $j=1,\dots,N$. Furthermore, $1\in\sigma(\fM)$ because $0\in\sigma(J_N)$ and thus there is a choice of constants $c_{1,1}$ and $c_{1,2}$ such that starting from these constants it is possible to reconstruct a function in $\ker(J_N)$. We conclude that $\fM=I$.
Hence, for any initial choice of the constants $c_{1,1}$ and $c_{1,2}$ we can reconstruct a function in $\ker(J_N)$ 
by recovering the coefficients $c_{j+1,1}$ and $c_{j+1,2}$, $j=1,2,\dots,N$ using the matrices in $\fM_j$ and the reconstruction is consistent thanks to $\fM = I$.
Thus, there are two degrees of freedom within $\ker(J_N)$ corresponding to the initial choice of the two constants $c_{1,1}$ and $c_{1,2}$  and item $(ii)$ is shown.
\end{proof}



\begin{rem}\label{rem:basis_for_zero2}
	If $0\in\sigma(J_N)$, it is possible to construct an orthonormal basis $\{\f_0^1,\f_0^2\}$ of $\ker J_N$ such that $\f_0^2 = \cS\f_0^1$ by applying the procedure described in Remark~\ref{rem:basis_for_zero}.
	\hfill$\diamondsuit$
\end{rem}

\begin{rem}
It is not possible to prove the simplicity of eigenvalues of $J_N$ in the intervals $(-1/2,0)$ and $(0,1/2)$, as in Proposition \ref{c-simple}. Indeed, double eigenvalues are observed in these intervals for a star-graph with $N=6$ edges described in Subsection \ref{ss-22}.
\hfill$\diamondsuit$
\end{rem}

We now construct an orthogonal decomposition of the Hilbert space $L^2_{\rm pol}(\RR^2,\CC^2) :=\bigoplus_{j=1}^NL^2_{\rm pol}(\Omega_j,\CC^2)$ in terms of the spectral decomposition of the spin-orbit operator $J_N$. For the sake of simplicity, the dependence of many objects on $N$ will be suppressed so that notation can be reused from Subsection \ref{ss-orthowedge}. Double eigenvalues of $J_N$ and the eigenvalue $0$ will again require careful consideration.

Let $\widetilde{\la}\in\sigma(J_N)\cap(0,\infty)$
with $n_{\widetilde\la} = 1$ and $\f_{\widetilde\la}^1\in\cF_{\widetilde\la}^N$ be normalized. Denote
\begin{align*}
\cE_{\widetilde\la}^1 := L^2_r(\RR_+)\otimes{\rm span}\{\f_{\widetilde\la}^1,\cS\f_{\widetilde\la}^1\}.
\end{align*}
Let $\widetilde{\la}\in\sigma(J_N)\cap(0,\infty)$
with $n_{\widetilde\la} = 2$ and
$\{\f^1_{\widetilde\la},\f^2_{\widetilde\la}\}$ be the orthonormal basis of
$\cF_{\widetilde\la}^N$. Denote
\begin{align*}
\cE_{\widetilde\la}^j :=L^2_r(\RR_+)\otimes{\rm span}\{\f_{\widetilde\la}^j,\cS\f_{\widetilde\la}^j\},\qquad j=1,2.
\end{align*}
For $0\in\sigma(J_N)$, let
$\{\f^1_{0},\f^2_{0}\}$ be the orthonormal basis of
$\cF_{0}^N$ constructed as in Remark~\ref{rem:basis_for_zero2} and define
\begin{align*}
\cE_{0} :=L^2_r(\RR_+)\otimes{\rm span}\{\f_{0}^1,\cS\f_{0}^1\}.
\end{align*}
The following orthogonal decomposition now holds if $0\notin\sigma(J_N)$
\begin{align*}
L^2_{\rm pol}(\RR^2,\CC^2)\simeq L^2_r(\RR_+)\otimes
\left[\bigoplus_{j=1}^NL^2(\II_j,\CC^2)\right]
=
\bigoplus_{
	\widetilde\la\in\sigma(J_N)\cap(0,\infty)}\bigoplus_{j=1}^{n_{\widetilde\la}}\cE_{\widetilde\la}^j,
\end{align*}
while if $0\in\sigma(J_N)$ we have a slightly modified decomposition
\begin{align*}
L^2_{\rm pol}(\RR^2,\CC^2)\simeq \cE_0\oplus\left(
\bigoplus_{
	\widetilde\la\in\sigma(J_N)\cap(0,\infty)}\bigoplus_{j=1}^{n_{\widetilde\la}}\cE_{\widetilde\la}^j\right).
\end{align*}
Finally, it is convenient to introduce the
unitary transformations $W_{\wt\la}^j\colon\cE_{\wt\la}^j\to L^2(\RR_+,\CC^2)$, for $\wt\la\in\sigma(J_N)\cap(0,\infty)$ and $1\le j\le n_{\wt\la}$, that act via
\begin{align}\label{eq:WNj}
(W_{\widetilde\la}^j u)(r):=\sqrt{r}
\begin{pmatrix}
\langle u(r,\cdot),\f_{\widetilde\la}^j\rangle_{\SS^1} \\
i\langle u(r,\cdot),\cS\f_{\widetilde\la}^j\rangle_{\SS^1}
\end{pmatrix}.
\end{align}
If $0\in\sigma(J_N)$, then the unitary mapping $W_0\colon\cE_0\rightarrow L^2(\RR_+,\CC^2)$ acts analogously via 
\begin{align*}
(W_{0} u)(r):=\sqrt{r}
\begin{pmatrix}
\langle u(r,\cdot),\f_{0}^1\rangle_{\SS^1} \\
i\langle u(r,\cdot),\cS\f_{0}^1\rangle_{\SS^1}
\end{pmatrix}.
\end{align*}
Also recall the half-line Dirac operators with off-diagonal Coulomb potentials denoted by $\bd_{\wt{\la}}$ given by equation \eqref{eq:1dDirac}. The operator $\fD_N$ can now be decomposed into an orthogonal sum of these operators $\bd_{\wt{\la}}$. The proof of the theorem is completely analogous to that of Theorem~\ref{p-decomp}.
\begin{theo}\label{p-decomp2}
	For $\wt\la\in\sigma(J_N)\cap(0,\infty)$, the operators
\begin{align*}
d_{\wt\la}^j u:=\widetilde{\cD}u, \hspace{3em} \dom d_{\wt\la}^j:=\dom\,\widetilde{\fD}_N\cap\cE_{\wt\la}^j,~~~\text{ for }~1\le j\le n_{\wt\la},
\end{align*}
and, in the case that $0\in\sigma(J_N)$, the operator
\begin{align*}
d_0 u:=\widetilde{\cD}u, \hspace{3em} \dom d_0:=\dom\,\widetilde{\fD}_N\cap\cE_0,
\end{align*}
are well defined in the Hilbert spaces $\cE_{\wt\la}^j$ and $\cE_0$, respectively. The operator $d_{\wt\la}^j$ is unitarily equivalent via $W_{\wt\la}^j$
	to the half-line Dirac operator $\bd_{\wt\la}$ in $L^2(\RR_+,\CC^2)$, while the operator $d_0$ is unitarily equivalent via $W_0$ to $\bd_0$.
	In particular, the decomposition
	\begin{align*}
	\fD_N\simeq
	\begin{cases}\bigoplus_{\wt\la\in
		\sigma(J_N)\cap (0,\infty)}\bigoplus_{j=1}^{n_{\wt\la}}\bd_{\wt\la},& 0\notin\sigma(J_N),\\
	\bd_0\oplus
	\left(\bigoplus_{\wt\la\in\sigma(J_N)\cap (0,\infty)}\bigoplus_{j=1}^{n_{\wt\la}}\bd_{\wt\la}\right),& 0\in\sigma(J_N),
	\end{cases}
	\end{align*}
	holds, the operator $\fD_N$ is closed and its deficiency indices can be computed as 
	\[
	n_{\pm}(\fD_N)=
	\begin{cases}
	\sum_{\wt\la\in\sigma(J_N)\cap(0,\infty)}
	n_{\wt\la}\cdot n_{\pm}(\bd_{\wt\la}),& 0\notin\sigma(J_N),\\
	n_\pm(\bd_0)+\sum_{\wt\la\in\sigma(J_N)\cap(0,\infty)}
	n_{\wt\la}\cdot n_{\pm}(\bd_{\wt\la}),& 0\in\sigma(J_N).
	\end{cases}
	\]
\end{theo}

The deficiency indices of $\fD_N$ can now be characterized by arguing analogously to Proposition~\ref{prop:def_indices_bl}.

\begin{prop}\label{prop:def_indices_bl2}
	The deficiency indices $n_\pm(\fD_N)$ of the Dirac operator $\fD_N$ are both equal to half of the number of eigenvalues of the spin-orbit operator $J_N$ lying in the interval $(-1/2,1/2)$ with multiplicities taken into account.
\end{prop}

In the most general case, the deficiency indices of $\fD_N$ can be estimated from above by using perturbation theory.

\begin{prop}\label{prop:def_bnd2}
	For any vector of angles $\{\w_j\}_{j=1}^{N-1}$ satisfying the above assumptions and for $\tau_j \in\RR\setminus\{-2,2\}$ for all $j\in\{1,2,\dots,N\}$, the operator $J_N$ has at most $2N$ eigenvalues in the interval $(-1/2,1/2)$, with multiplicities taken into account. In particular, the deficiency indices of the Dirac operator $\fD_N$ satisfy $n_+(\fD_N) = n_-(\fD_N) \le N$.
\end{prop}
\begin{proof}
	Consider the symmetric operator
	\[
		A\varphi := -i\sigma_3\varphi' +\frac{\varphi}{2},\qquad\dom A := \bigoplus_{j=1}^N H^1_0(\II_j,\CC^2), 
	\] 
	which is closed and densely defined in the Hilbert space $L^2(\SS^1,\CC^2)$. The operator $H^1_0(\II_j)\ni\varphi\mapsto -i\varphi'$ acting in the Hilbert space $L^2(\II_j)$ has deficiency indices $(1,1)$, so the deficiency indices of the operator $A$ can be identified as $(2N,2N)$.
	
Consider the self-adjoint extension of the symmetric operator $A$ given by
	\[
		B\varphi := -i\sigma_3\varphi' +\frac{\varphi}{2},\qquad\dom B := H^1(\mathbb{S}^1,\CC^2).
	\]  
	The spectrum of $B$ can be explicitly computed as $\sigma(B) = \ZZ +\frac12$ and, in particular, $\sigma(B)\cap (-\frac12,\frac12) = \varnothing$. The spin-orbit operator $J_N$ is also a self-adjoint extension of $A$ though, so we conclude from~\cite[Satz 10.18]{W00} that
	the number of eigenvalues of $J_N$ in the interval $(-1/2,1/2)$, with multiplicities taken into account, is at most $2N$. Finally, the upper bound on the deficiency indices for $\fD_N$ follows by applying Proposition~\ref{prop:def_indices_bl2}.
\end{proof}

\subsection{Example: star-graphs with three edges}\label{ss-threeleads}

The explicit analysis of the spin-orbit operator in Subsections \ref{ss-twoleads} and \ref{ss-twoleads2} is not feasible for general star-graphs. In this subsection, however, we lay the groundwork for such an analysis when the star-graph has $N=3$ edges. This analysis will be sufficient to determine deficiency indices of symmetric star-graphs with special configurations of interaction parameters in Subsection \ref{ss-threeequalleads}.

For simplicity, we denote $\w_1=\w$. General eigenfunctions must satisfy boundary conditions from the matrices $M_j$, for $j=1,2,3$, which can be written down by using equation \eqref{e-Nmatrixequations}. The notation $p_j$ and $m_j$ will denote the obvious analogs of equation \eqref{e-pm} for $j\in\{1,2,3\}$. Explicitly, eigenfunctions must satisfy:
\begin{equation}\label{e-3matrixequations}
    \begin{aligned}
    c_{1,1}&=c_{3,1}\frac{m_1e^{i2\pi\la}}{p_1}-c_{3,2}\frac{\tau_1e^{i[\w(2\la+1)-2\pi\la]}}{p_1}, \\
    c_{1,2}&=-c_{3,1}\frac{\tau_1e^{-i[\w(2\la+1)-2\pi\la]}}{p_1}+c_{3,2}\frac{m_1e^{-i2\pi\la}}{p_1}, \\
    c_{1,1}&=c_{2,1}\frac{m_2}{p_2}+c_{2,2}\frac{\tau_2e^{-i\w(2\la+1)}}{p_2}, \\
    c_{1,2}&=c_{2,1}\frac{\tau_2e^{i\w(2\la+1)}}{p_2}+c_{2,2}\frac{m_2}{p_2}, \\ 
    c_{2,1}&=c_{3,1}\frac{m_3}{p_3}+c_{3,2}\frac{e^{-i\w_2(2\la+1)}\tau_3}{p_3}, \\
    c_{2,2}&=c_{3,1}\frac{e^{i\w_2(2\la+1)}\tau_3}{p_3}+c_{3,2}\frac{m_3}{p_3}.
\end{aligned}
\end{equation}
Note that the first two equations do not match the two equations given by equation \eqref{e-Nmatrixequations2} but are easily shown to be equivalent. This system of equations can be used to find the eigenvalues of the operator $J_3$. 

\begin{prop}
Let $\tau_j\neq\pm2$ for any $j\in\{1,2,3\}$ and $J_3$ be given by equation \eqref{e-angular3}. Then,  $\la$ is an eigenvalue of $J_3-1/2$
if and only if it satisfies the equation
\begin{equation}\label{e-3eigenvalues}
\begin{aligned}
0=p_1p_2p_3&-m_1m_2m_3\cos(2\pi\la)-m_1\tau_2\tau_3\cos[(\w_2-\w)(2\la+1)-2\pi\la] \\
&-m_2\tau_1\tau_3\cos[-(\w+\w_2)(2\la+1)+2\pi\la]-m_3\tau_1\tau_2\cos[2\w(2\la+1)-2\pi\la].
\end{aligned}
\end{equation}
\end{prop}

\begin{proof}
The system of equations given by  \eqref{e-3matrixequations} has a solution if and only if the following determinant vanishes
\begin{align*}
    \begin{vmatrix}
    1 & 0 & 0 & 0 & -\frac{m_1e^{i2\pi\la}}{p_1} & \frac{\tau_1e^{i[\w(2\la+1)-2\pi\la]}}{p_1} \\
    0 & 1 & 0 & 0 & \frac{\tau_1e^{-i[\w((2\la+1)-2\pi\la]}}{p_1} & -\frac{m_1e^{-i2\pi\la}}{p_1} \\
    1 & 0 & -\frac{m_2}{p_2} & -\frac{\tau_2e^{-i\w(2\la+1)}}{p_2} & 0 & 0 \\
    0 & 1 & -\frac{\tau_2e^{i\w(2\la+1)}}{p_2} & -\frac{m_2}{p_2} & 0 & 0 \\
    0 & 0 & 1 & 0 & -\frac{m_3}{p_3} & -\frac{\tau_3e^{-i\w_2(2\la+1)}}{p_3} \\
    0 & 0 & 0 & 1 & -\frac{\tau_3e^{i\w_2(2\la+1)}}{p_3} & -\frac{m_3}{p_3}
    \end{vmatrix}.
\end{align*}
A tedious calculation yields the desired expression.
\end{proof}

There are several consequences that can immediately be drawn from equation \eqref{e-3eigenvalues}. First, if $\tau_3=0$ then the equation collapses and is identical to equation \eqref{e-2eigenvalues}, representing the broken line after replacing $\tau_2$ by $-\tau_2$. This replacement is needed because the convention for the direction of the normal vectors was different for the setting of the broken line. Second, the contributions of the parameters $\w$ and $\w_2$ are conveniently represented: the values $2\w$, $\w_2-\w$ and $2\pi-\w-\w_2$ are the angles of the three sectors, respectively. Finally, despite the somewhat simple looking nature of equation \eqref{e-3eigenvalues}, we will be forced again to analyze more special cases in order to recover concrete results.

\subsection{Example: symmetric star-graph with three edges}\label{ss-threeequalleads}

This subsection further analyzes Dirac operators with singular interactions on star-graphs with $N=3$ edges with particular configurations of parameters. In particular, we attempt to determine the number of eigenvalues
for $J_3$, with multiplicities taken into account, lying in the interval $(-1/2,1/2)$. This knowledge will allow us to invoke Proposition~\ref{prop:def_indices_bl2} and compute the deficiency indices of $\fD_3$ with the chosen parameters.

A natural example is when all three edges of the star-graph are equally spaced in terms of angles. Hence, let $\w=\pi/3$ and $\w_2=\pi$ so that each of the three sectors has an opening angle of $2\pi/3$. Algebraic difficulties can be further simplified by letting $\tau_1 = \tau_2 = \tau_3 = \tau$ so that the singular interactions strengths on each edge are equal. Equation \eqref{e-3eigenvalues} then reduces to 
\begin{align}\label{e-threeequal}
    (4+\tau^2)^3\cos(2\pi\la)-(4-\tau^2)^3=48(4+\tau^2)\tau^2\cos\left[\frac{\pi}{3}(2\la+1)\right].
\end{align}

This equation \eqref{e-threeequal} can now be analyzed in the same way as equations \eqref{e-2equal} and \eqref{e-2opposite} from Subsections \ref{ss-twoleads} and \ref{ss-twoleads2}, respectively.

\begin{prop}\label{p-3equal}
Let the above assumptions hold and $\tau\neq\pm 2$. If $\tau<-2\sqrt{3}$ or $\tau>2\sqrt{3}$, then there exists exactly one $\la_1\in(-1,-1/2)$ and one $\la_2\in(-1/2,0)$ that satisfy equation~\eqref{e-threeequal}, while
$\la =-1/2$ does not satisfy~\eqref{e-threeequal}.
If $\tau\in[-2\sqrt{3},2\sqrt{3}]$, then there are no $\la\in(-1,0)$ satisfying~\eqref{e-threeequal}.
\end{prop}

\begin{proof}
Analysis of equation \eqref{e-threeequal} is surprisingly complicated, so for the sake of brevity we include only the solutions and relevant conclusions. The solutions of equation \eqref{e-threeequal}
with respect to $\la$ are given for all $k\in\ZZ$ by
\begin{align}\label{e--4/3}
\la=
\begin{cases}
    3k+1/2  \\
    3k+3/2  \\
    3k-3\pi^{-1}\arctan\left(\frac{2+\sqrt{5}}{\sqrt{3}}\right)  \\
    3k+3\pi^{-1}\arctan\left(\frac{\sqrt{5}-2}{\sqrt{3}}\right) 
    \end{cases}
    \text{ for }\tau=\pm \frac{2}{\sqrt{3}},
\end{align}
and
\begin{align}\label{e-othersolns}
\la=
\begin{cases}
    3k+3\pi^{-1}\arctan\left(
    \frac{6-\sqrt{3}\tau}{3\tau+2\sqrt{3}}\right)  \\
    3k+3\pi^{-1}\arctan\left(\frac{6+\sqrt{3}\tau}{2\sqrt{3}-3\tau}\right)  \\
    3k-3\pi^{-1}\arctan\left(\frac{12+3\tau^2+\sqrt{3}\sqrt{48+40\tau^2+3\tau^4}}{8\sqrt{3}}\right) \\
    3k+3\pi^{-1}\arctan\left(\frac{-12-3\tau^2+\sqrt{3}\sqrt{48+40\tau^2+3\tau^4}}{8\sqrt{3}}\right)
    \end{cases}
    \text{ for }\tau\neq\pm 2,\pm \frac{2}{\sqrt{3}}.
\end{align}

The solutions in equation \eqref{e--4/3} do not allow for $\la\in(-1,0)$ for any choice of $k\in\ZZ$. Equation \eqref{e-othersolns} can  produce solutions in the interval $(-1,0)$ only for $k=0$. The first two solutions in equation \eqref{e-othersolns} with $k=0$ yield one solution in the interval $(-1,-1/2)$ and another solution in the interval $(-1/2,0)$, respectively, only when $\tau<-2\sqrt{3}$ or $\tau>2\sqrt{3}$. For $-2\sqrt{3}<\tau <2\sqrt{3}$, none of these solutions yields a root in the interval $(-1,0)$. The third and fourth solutions of \eqref{e-othersolns} do not fall in the interval $(-1,0)$ for any values of $\tau$. 
\end{proof}

It remains to check the multiplicity of the eigenvalues of $J_3$ within the interval $(-1/2,1/2)$, as this may increase the deficiency indices of $\fD_3$. We abbreviate $m_1=m_2=m_3=m$ and $p_1=p_2=p_3=p$ to simplify notation in the proof.

\begin{prop}\label{p-threemult}
	Let the above assumptions hold and $\tau\ne\pm 2$.
	The eigenvalues of $J_3$ within the interval $(-1/2,1/2)$ are simple.
\end{prop}

\begin{proof}
	The proof is analogous to that of Proposition \ref{c-simple}, with the exception that we now rely on the conditions stated in equation \eqref{e-3matrixequations}. Substitute the last two equations into the third and fourth, respectively. Eliminate the constants $c_{1,1}$ and $c_{1,2}$ by setting this new pair of equations equal to the first two equations in equation \eqref{e-3matrixequations}. This results in two conditions, which can be written as $c_{3,1}f(\tau,\la)=-c_{3,2}g(\tau,\la)$ and $c_{3,1}\widetilde{f}(\tau,\la)=-c_{3,2}\widetilde{g}(\tau,\la)$, where
	\begin{equation}\label{e-threemultcond}
	\begin{aligned}
	f(\tau,\la)&=\frac{m^2}{p^2}+\frac{\tau^2}{p^2}e^{i(2\la+1)(2\pi/3)}-\frac{m}{p}e^{2i\pi\la}, \\
	g(\tau,\la)&=\frac{m\tau}{p^2}\left(e^{-i\pi(2\la+1)}+e^{-i(2\la+1)(\pi/3)}\right)+\frac{\tau}{p}e^{i(-4\pi\la/3+\pi/3)}, \\
	\widetilde{f}(\tau,\la)&=\frac{m\tau}{p^2}\left(e^{i\pi(2\la+1)}+e^{i(2\la+1)(\pi/3)}\right)+\frac{\tau}{p}e^{-i(-4\pi\la/3+\pi/3)}, \\
	\widetilde{g}(\tau,\la)&=\frac{m^2}{p^2}+\frac{\tau^2}{p^2}e^{-i(2\la+1)(2\pi/3)}-\frac{m}{p}e^{-2i\pi\la}.
	\end{aligned}
	\end{equation}
	In order for an eigenvalue to have multiplicity two, all four of the functions in equation \eqref{e-threemultcond} must be equal to $0$, thereby giving eigenfunctions two degrees of freedom. In particular, this means that we can set $\widetilde{f}(\tau,\la)=g(\tau,\la)$ so that exponentials can be combined. Simplification yields that a necessary condition for an eigenvalue to have multiplicity two is
	\begin{align}\label{e-3mult2}
	\frac{p}{m}\sin\left(-\frac{4\pi}{3}\la+\frac{\pi}{3}\right)=\sin\left[\frac{\pi}{3}(2\la+1)\right]+\sin[\pi(2\la+1)].
	\end{align}
	For $\la\in(-1/2,0)$ and $|\tau|\geq2\sqrt{3}$, the left-hand side of equation \eqref{e-3mult2} is negative while the right-hand side is a sum of two positive numbers. As eigenvalues of $J_3-1/2$ with multiplicities taken into account are symmetric about $\la=-1/2$, and $\la=-1/2$ is not an eigenvalue by Proposition \ref{p-3equal}, we conclude that the spectrum of $J_3-1/2$ is simple within the interval $(-1,0)$. 
\end{proof}
Multiplicity of eigenvalues of $J_3$ within the interval $(-1/2,1/2)$ thus accounted for, it is possible to state the deficiency indices of $\fD_3$ for the given setup.

\begin{cor}\label{c-3defindices}
Let $\w=\pi/3$, $\w_2=\pi$ and $\tau_1=\tau_2=\tau_3 = \tau\neq\pm2$. If $\tau<-2\sqrt{3}$ or $\tau>2\sqrt{3}$, then the deficiency indices of $\fD_3$ are $n_+(\fD_3) = n_-(\fD_3)=(1,1)$. If $\tau\in[-2\sqrt{3},2\sqrt{3}]$, then the
operator $\fD_3$ is self-adjoint.
\end{cor}

\begin{proof}
The result follows by combining Propositions~\ref{prop:def_indices_bl2}, \ref{p-3equal}, and~\ref{p-threemult}.
\end{proof}

\subsection{Alternative spectral condition}\label{ss-altspec}

As shown in Proposition~\ref{prop:def_indices_bl2}, computing the deficiency indices of the Dirac operator $\fD_N$ corresponding to the star-graph with $N$ leads boils down to determining how many eigenvalues of the operator $J_N$, counting multiplicities, are in the interval $(-1/2,1/2)$. Instead of determining a system of equations that yields conditions on generic eigenfunctions of $J_N$ and finding when this system has a non-trivial solution, as elsewhere in the manuscript, we can further analyze the momentum operator $J_N -1/2$
on the graph discussed in Proposition~\ref{p-sathree}. Recall that the respective graph has $2N$ directed edges and $N$ vertices (given by $\w_{j-1}$ for $j=1,\dots,N$) and $J_N -1/2$ operates as the differential expression $-i\frac{d}{dx}$ on each edge. Each vertex has two incoming and two outgoing edges, one to each of the neighboring vertices. See the proof of Proposition~\ref{p-sathree} for more details on this setup.

The key object from this construction is the unitary matrix $\cU_N$ which realizes
the boundary condition $\cU_N\psi\ti{in}=\psi\ti{out}$ and whose entries are given explicitly in equation \eqref{e-Uentries}. In the special case where the angles $\w_j$ are chosen so that $\RR^2$ is split into equal sectors, the spectral analysis of $J_N-1/2$ can be reduced to finding the eigenvalues of $\cU_N$. In this symmetric case, the operator $J_N - 1/2$ can be identified with the following operator acting in the Hilbert space $\bigoplus_{j=1}^{2N} L^2(0,2\pi/N)$ via
\begin{equation}\label{eq:momentum_J}
	\left
	\{\psi=\oplus_{j=1}^{2N}\psi_j\in \bigoplus_{j=1}^{2N}H^1(0,2\pi/N)\colon 
	\cU_N\psi(2\pi/N) = \psi(0)\right\}\mapsto \bigoplus_{j=1}^{2N}(-i\psi_j').
\end{equation}

This interpretation is inspired by \cite{E}, and the proof of the following Proposition is a slight generalization of the first step of the proof \cite[Theorem 5.4]{E}, as it takes multiplicities into account.

\begin{prop}\label{p-altspec}
Let $\w_j-\w_{j-1}=2\pi/N$ for all $j=1,\dots,N$. Then $\la$ is an eigenvalue of $J_N-1/2$ if and only if $e^{-2i\pi\la/N}$ is an eigenvalue of $\cU_N$
and the multiplicities of these eigenvalues coincide.
\end{prop}

\begin{proof}
We claim that the mapping
\[
	\cR\colon\CC^{2N} \rightarrow 
	\bigoplus_{n=1}^{2N} L^2(0,2\pi/N),\qquad 
	\cR{\bf c} = \bigoplus_{j=1}^{2N} c_je^{i\la\te},\qquad {\bf c}= (c_1,c_2,\dots, c_{2N})^\top.
\]
is a bijection between $\ker(\cU_N - e^{-2i\pi\la/N})$ and $\ker(J_N - 1/2-\la)$, where the operator $J_N-1/2$ is interpreted as in~\eqref{eq:momentum_J}.

Let ${\bf c}\in\ker(\cU_N - e^{-2i\pi\la/N})$. Then, $-i(\cR{\bf c})' = \la\cR{\bf c}$ and the boundary condition 
\[
\cU_N((\cR{\bf c})(2\pi/N)) = e^{i2\pi\la/N}\cU_N((\cR{\bf c})(0))= \cR{\bf c}(0)
\]
holds. Hence, we conclude that $\cR{\bf c}\in\ker(J_N - 1/2-\la)$. Since $\ker\cR = \{0\}$, the mapping $\cR$ is injective and it only remains to check that $\cR$ maps $\ker(\cU_N - e^{-2i\pi\la/N})$ onto $\ker(J_N -1/2-\la)$. Let $\psi \in \ker(J_N - 1/2 - \la)$ be arbitrary. Then, $\psi(0) = e^{-i2\pi\la/N}\psi(2\pi\la/N)$ and the boundary condition $\cU_N\psi(0) = \psi(2\pi\la/N)$ imply that $\psi(0)\in \ker(\cU_N - e^{-2i\pi\la/N})$. Finally, this means that $\cR(\psi(0)) = \psi$ and $\cR$ is surjective.
\end{proof}

Proposition \ref{p-altspec} says that when the star-graph is symmetric instead of searching for those $\la\in(-1,0)$ that are eigenvalues of $J_N-1/2$, we can search for eigenvalues of $\cU_N$ that fall on the arc $\cA_N := \{z\in\TT\colon 0 < {\rm arg}\,(z) < 2\pi/N\}$ of the unit circle. The deficiency indices of $\fD_N$ are then equal to half of the number of eigenvalues of the matrix $\cU_N$ on the arc $\cA_N$ with multiplicities taken into account.
In practice, the spectrum of $\cU_N$ is much easier to determine than by using the general methods of Section~\ref{s-star} and allows for more sophisticated examples to be analyzed.

\subsection{Further examples and comments}\label{ss-22}

Thanks to the alternative spectral condition given in Proposition \ref{p-altspec} and the explicit entries of $\cU_N$ given in equation \eqref{e-Uentries}, it is possible to compute the eigenvalues of $J_N-1/2$ for higher values of $N$ in the special case when $\w_j-\w_{j-1}=2\pi/N$ for all $j=1,\dots,N$ and the other parameters are given. We present several such cases that are of interest. Although the results are primarily numerical, they can be easily verified.

Let $N=6$. Following Proposition ~\ref{p-altspec}, the edges of the star-graph are chosen to divide $\RR^2$ into $6$ equal sectors, each with opening angle $\pi/3$. If, additionally, the interaction parameters are chosen to be 
\begin{align}\label{e-6doubles}
\tau_1=\tau_3=\tau_4=\tau_6=1 \text{ and } \tau_2=\tau_5=-1
\end{align}
then we find that there are two double eigenvalues of the matrix $\cU_6$ lying on the arc $\cA_6$. These eigenvalues are approximately given by $0.976136+0.217162i$ and $0.676136+0.736777i$, see Figure \ref{f-6double}. Hence, for this configuration $\fD_6$ has deficiency indices $(2,2)$. Surprisingly, with these interaction parameters, all eigenvalues of $\cU_6$ are double. 

\begin{figure}[ht]\label{f-6double}
\includegraphics[scale=.4]{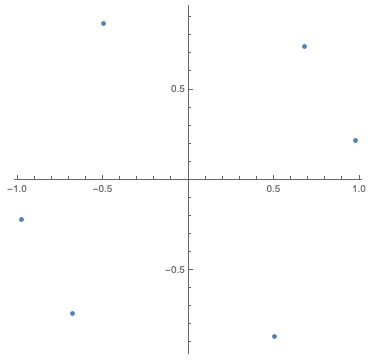}
   \caption{The eigenvalues of $\cU_6$ with parameters given by equation \eqref{e-6doubles}.}
   \centering
\end{figure}

It can be inferred that alternating the interaction strengths on each edge between positive and negative values may be the crucial element which gives rise to these multiple eigenvalues and increased deficiency indices. After all, the explicit analyses for $N=2,3$ did not cover such a property and the conclusions were that eigenvalues were simple and deficiency indices were $(1,1)$ or lower. However, configurations similar to the above example, where interaction strengths alternated, did not lead to similar results when $N=4$ or $N=5$. Indeed, the case $N=6$ was the lowest value for which we were able to observe $4$ eigenvalues, counting multiplicities, lying on the arc $\cA_N$, corresponding to $\fD_N$ having deficiency indices $(2,2)$. Of course, it may be possible to have deficiency indices higher than $(1,1)$ for lower values of $N$, but tests to find such eigenvalues were unsuccessful. 

It is also possible to observe multiple transitions for the deficiency indices by changing only two interaction strengths. For example, let $\tau_2=\tau_4=\tau_5=\tau_6=1$ and denote $t:=\tau_1=\tau_3$. It is possible to conjecture after many numerical tests that for large and small values of $t$, examples where $t<-24$ or $t>-1/6$, that $\fD_6$ is self-adjoint. Furthermore, examples where $t\in(-24,-\sqrt{21})\cup(-7/8,-1/6)$ found deficiency indices $(1,1)$ and $t\in(-\sqrt{21},-7/8)$ found deficiency indices $(2,2)$. Of course, the transition points are only rough estimates, but nevertheless the range of possibilities from varying $t$ is interesting.

\section{Characterization of self-adjoint extensions}\label{s-distinguished}

The Dirac operator $\fD_N$ presented in Section \ref{s-star} has equal deficiency indices by Proposition~\ref{prop:def_indices_bl2}, so it possesses a family of self-adjoint extensions. These extensions can be parametrized with the help of the orthogonal decomposition from Theorem~\ref{p-decomp2} and classical von Neumann extension theory; see e.g.~\cite[\S 13.2]{S12} and~\cite[\S 6.1]{EE}. In applications, it is important to fix a specific self-adjoint extension to work with and there are many ways of doing this depending on the context. One common way to choose such an extension from the parametrized family is to find the most regular operator domain, in terms of the scale of Sobolev spaces. In the literature, such an extension is usually referred to as the `distinguished' self-adjoint extension. Under the assumption $0\notin\sigma(J_N)$, it is possible for us to identify a unique distinguished self-adjoint extension by requiring that its operator domain is contained in the Sobolev space $\bigoplus_{j=1}^N H^{1/2}(\Omega_j,\CC^2)$. 

The first step to achieving a parametrization of all self-adjoint extensions is to explicitly find defect elements for the operator $\fD_N$. Theorem~\ref{p-decomp2} shows that deficiency indices of $\fD_N$ come solely from the operators $\bd_{\wt\la}$ in its orthogonal decomposition related to the eigenvalues $\wt\la$ of $J_N$ lying in the interval $(-1/2,1/2)$.  All other operators in the orthogonal decomposition of $\fD_N$ are self-adjoint. 

The next lemma characterizes the defect subspaces of $\bd_{\wt\la}$ for $\wt\la \in [0,1/2)$. This analysis is partially carried out in \cite[Equation (1.25)]{CP} and \cite[Proof of Lemma 2.5]{TOB} but we provide some details here for the convenience of the reader.

\begin{lem}\label{p-decompdefelements}
Let the symmetric operator $\bd_{\wt\la}$ for $\wt\la \in [0,1/2)$,	acting in the Hilbert space $L^2(\RR_+,\CC^2)$ be as in equation \eqref{eq:1dDirac}.
Then the defect subspaces are characterized by $\Ker(\bd^*_{\wt\la}\pm i) = {\rm span}\,\{f_{\wt\la}^\pm\}$ where
	\begin{align*}
	    f_{\wt\la}^+ :=
	    \begin{pmatrix} 
	    r^{1/2}K\ci{\widetilde{\la}-1/2}(r) \\ 
	    -ir^{1/2}K\ci{\widetilde{\la}+1/2}(r) 
	    \end{pmatrix}
	    ,\hspace{2em}
	    f_{\wt\la}^-:=
	    \begin{pmatrix} 
	    r^{1/2}K\ci{\widetilde{\la}-1/2}(r) \\ 
	    ir^{1/2}K\ci{\widetilde{\la} + 1/2}(r) 
	    \end{pmatrix},
	\end{align*}
and $K_\nu(\cdot)$ is the modified Bessel function of the second kind and order $\nu\in\RR$. 
\end{lem}

\begin{proof}
We will first characterize $\ker(\bd_{\wt\la}^*+i)$. Simplification gives
\begin{equation}
\label{e-posdefect}
\begin{pmatrix}
    f_1 \\
    f_2
    \end{pmatrix}
    \in\Ker(\bd^*_{\wt\la} +i)\quad\iff\quad
    \begin{cases}
    if_1+\left(-\dfrac{d}{dr}-\dfrac{\widetilde{\la}}{r}\right)f_2=0, \\
    \left(\dfrac{d}{dr}-\dfrac{\widetilde{\la}}{r}\right)f_1+if_2=0.
    \end{cases}
\end{equation}
Solving the second equation for $f_2$ and substituting back into the first equation yields the differential equation
\begin{align*}
    f_1''+\left(-1-\dfrac{\widetilde{\la}^2-\widetilde{\la}}{r^2}\right)f_1=0.
\end{align*}
The square-integrable solution is given up to multiplication by a constant by \cite[Equation 10.13.1]{DLMF} and rewritten by \cite[Equation 10.27.8]{DLMF} as
\begin{align*}
    f_1=r^{1/2}K\ci{\widetilde{\la}-1/2}(r).
\end{align*}
Plugging this back into the system of differential equations in~\eqref{e-posdefect} and using \cite[Equation 10.29.2]{DLMF} yields 
\begin{align*}
    f_2=-ir^{1/2}K\ci{\widetilde{\la}+1/2}(r).
\end{align*}
Characterization of $\ker(\bd_{\wt\la}^*-i)$ proceeds similarly, with the same differential equation providing for $f_1$, but with an additional factor of $-1$ present in $f_2$. 
\end{proof}

Lemma~\ref{p-decompdefelements} immediately allows us to describe all possible self-adjoint extensions of the operator $\fD_N$. To this end, let $\cD_N^{\pm}=\ker(\fD_N^*\mp i)$ denote the positive and negative defect spaces of $\fD_N$. Also recall that $\f_{\wt\la}^j$, $1\le j\le n_{\wt\la}$
is the orthonormal basis of $\ker(J_N-\wt\la)$ for $\wt\la \in \sigma(J_N)\cap (0,\infty)$ of multiplicity $n_{\wt\la}\in\{1,2\}$. The orthonormal basis of $\ker J_N$ is $\{\f_0^1,\f_0^2\}$ and is constructed as in Remark~\ref{rem:basis_for_zero2}. 

\begin{theo}\label{t-nbreakdown}
	Let the angles $\{\w_j\}_{j=1}^{N-1}$ and the interaction strengths $\{\tau_j\}_{j=1}^N$ be chosen as in Section~\ref{s-star} so that the operator $\fD_N$ has deficiency indices $n_+(\fD_N)=n_-(\fD_N)=n\ge1$. All self-adjoint extensions of $\fD_N$ are in one-to-one correspondence with unitary $n\times n$ matrices $\cU=\{u_{m,k}\}_{k,m=1}^{n,n}$. 
	If $0\notin\sigma(J_N)$, the self-adjoint extension $\fD_{N}^\cU$ of $\fD_N$ corresponding to the matrix $\cU$ is characterized by
	\[
	\begin{aligned}
	    \dom\fD_N^{\cU} &= 
	    \bigg\{w=u+v_+ + v_-\in L^2(\RR^2,\CC^2)\colon u\in \dom\fD_N,\\
	    &\qquad\quad v_+ = \sum_{\wt\la\in\sigma(J_N)\cap (0,\frac12)}\sum_{j=1}^{n_{\wt\la}} c_{k(\wt\la)+j}\big[K_{\wt\la-1/2}(r)\f_{\wt\la}^j(\te) + K_{\wt\la+1/2}(r)\cS\f_{\wt\la}^j(\te)\big]\\
	    	    &\qquad\quad v_- = \sum_{\wt\la\in\sigma(J_N)\cap (0,\frac12)}\sum_{j=1}^{n_{\wt\la}} (\cU {\bf c})_{k(\wt\la)+j}\big[K_{\wt\la-1/2}(r)\f_{\wt\la}^j(\te) - K_{\wt\la+1/2}(r)\cS\f_{\wt\la}^j(\te)\big]  \bigg\}
	      \\
	      \fD^\cU_Nw &=\fD_N^{\cU}(u+v_+ + v_-) = \fD_N u + iv_+ -i v_-,
	\end{aligned}
	\]
	where ${\bf c} = (c_1,c_2,\dots,c_n)\in\CC^n$ is an arbitrary vector with complex entries
	and $k(\wt\la)$ is the number of eigenvalues of $J_N$ in the interval $(0,\wt\la)$ counted with multiplicities. 
 If $0\in\sigma(J_N)$, the self-adjoint extension $\fD_{N}^\cU$ of $\fD_N$ corresponding to the matrix $\cU$ is characterized by
	\[
	\begin{aligned}
	\dom\fD_N^{\cU} &= 
	\bigg\{w=u+v_+ + v_-\in L^2(\RR^2,\CC^2)\colon u\in \dom\fD_N,\\
	&\qquad v_+ =
	c_1\big[K_{1/2}(r)\f_0^1(\te) + K_{1/2}(r)\cS\f_0^1(\te)\big]\\
	&\qquad\qquad+
	 \sum_{\wt\la\in\sigma(J_N)\cap (0,\frac12)}\sum_{j=1}^{n_{\wt\la}} c_{k(\wt\la)+j+1}\big[K_{\wt\la-1/2}(r)\f_{\wt\la}^j(\te) + K_{\wt\la+1/2}(r)\cS\f_{\wt\la}^j(\te)\big]\\
	&\qquad v_- =
	(\cU {\bf c})_1\big[K_{1/2}(r)\f_0^1(\te) + K_{1/2}(r)\cS\f_0^1(\te)\big]\\
	&\qquad\qquad+
	\sum_{\wt\la\in\sigma(J_N)\cap (0,\frac12)}\sum_{j=1}^{n_{\wt\la}} (\cU {\bf c})_{k(\wt\la)+j+1}\big[K_{\wt\la-1/2}(r)\f_{\wt\la}^j(\te) - K_{\wt\la+1/2}(r)\cS\f_{\wt\la}^j(\te)\big]  \bigg\}
\\
	\fD^\cU_Nw &=\fD_N^{\cU}(u+v_+ + v_-) = \fD_N u + iv_+ -i v_-,
	\end{aligned}
	\]
	where the vector ${\bf c}\in\CC^n$ 
	and the function $k(\wt\la)$ are as before.
\end{theo}

\begin{proof}
Assume $0\notin\sigma(J_N)$. The case $0\in\sigma(J_N)$ will make use of the identity  $K_{\nu}(z)=K_{-\nu}(z)$, but can otherwise be treated analogously. The defect elements of $\bd_{\wt\la}$ for $\wt\la\in[0,1/2)$, stated in Lemma \ref{p-decompdefelements}, can be mapped via the inverses of the unitary operators in~\eqref{eq:WNj} that yield the orthogonal decomposition:
	\begin{align*}
	    [(W_{\wt\la}^j)^{-1}f^+_{\wt\la}]
	    (r,\te)&=
	    K\ci{\wt\la-1/2}(r)\f^j_{\wt\la}(\te)+K\ci{\wt\la+1/2}(r)\cS\f^j_{\wt\la}(\te)\in\cD_N^+, \\
	    [(W_{\wt\la}^j)^{-1}f^-_{\wt\la}]
	    (r,\te)&=
	    K\ci{\wt\la-1/2}(r)\f^j_{\wt\la}(\te)-K\ci{\wt\la+1/2}(r)\cS\f^j_{\wt\la}(\te)\in\cD_N^-.
	\end{align*}
	These pre-images form orthogonal bases for $\cD_N^+$ and $\cD_N^-$.
	The matrix $\cU$ defines a unitary mapping
	\[
	\begin{aligned}
		\cD_N^+&\ni   \sum_{\wt\la\in\sigma(J_N)\cap (0,\frac12)}\sum_{j=1}^{n_{\wt\la}} c_{k(\wt\la)+j}\big[K_{\wt\la-1/2}(r)\f_{\wt\la}^j(\te) + K_{\wt\la+1/2}(r)\cS\f_{\wt\la}^j(\te)\big]\\
		&\qquad\qquad \mapsto
		\sum_{\wt\la\in\sigma(J_N)\cap (0,\frac12)}\sum_{j=1}^{n_{\wt\la}} (\cU {\bf c})_{k(\wt\la)+j}\big[K_{\wt\la-1/2}(r)\f_{\wt\la}^j(\te) - K_{\wt\la+1/2}(r)\cS\f_{\wt\la}^j(\te)\big]\in \cD_N^-.
	\end{aligned}
	\]
	 All self-adjoint extensions of $\fD_N$ are thus in one-to-one correspondence with unitary $n\times n$ matrices $\cU$ by the von Neumann decomposition, see e.g.~\cite[Equation (6.1.7)]{EE} and also~\cite[Theorem 13.10]{S12}. This decomposition also yields the characterization of the action and operator domain for $\fD^\cU_N$.
\end{proof}

	\begin{rem}
		As in~\cite[Remark 1.5]{TOB}, the functions $K_\nu(r)$, $\nu\in(0,1)$,
		in the characterization of the operator domain of $\fD_N^\cU$ can be replaced by $r^{-\nu}\chi(r)$ where $\chi\colon\R_+\rightarrow[0,1]$ is a $C^\infty$-smooth function that equals $1$ in a neighbourhood of $0$ and $0$ for $r$ large enough. \hfill$\diamondsuit$ 		 
	\end{rem}
	
The decomposition of general functions in the domain $\fD_N^\cU$ given in Theorem \ref{t-nbreakdown} provides a tool to classify the regularity of functions in the domain. When $0\notin\sigma(J_N)$, it is possible to single out a unique distinguished self-adjoint extension whose operator domain is the most regular in the scale of Sobolev spaces. A similar statement is made for Dirac operators with infinite mass boundary conditions in sectors in \cite[Theorem 1.2(ii)]{TOB}.

\begin{theo}\label{t-distinguished}
	Let the angles $\{\w_j\}_{j=1}^{N-1}$ and the interaction strengths $\{\tau_j\}_{j=1}^N$ be chosen as in Section~\ref{s-star} so that the operator $\fD_N$ has deficiency indices $n_+(\fD_N) = n_-(\fD_N) = n\ge1 $ and $0\notin\sigma(J_N)$. The so called `distinguished' self-adjoint extension of $\fD_N$, the only extension such that 
	\begin{align*}
	    \dom\fD_N^{\cU}\subset\bigoplus_{j=1}^N H^{1/2}(\Omega_j,\CC^2),
	\end{align*}
	is the extension corresponding to $\cU=\cI_n$, the $n\times n$ identity matrix. In fact, the following slightly stronger statement holds
	\begin{align*}
	\dom\fD_N^{\cI_n}\subset\bigoplus_{j=1}^N H^s(\Omega_j,\CC^2),
	\end{align*}
	for all $s<1/2+\min\{\sigma(J_N)\cap(0,\infty)\}$.
\end{theo}

Before beginning the proof, we recall a few facts about modified Bessel functions, which will play a key role.
In particular, the Bessel function $K_\nu(z)$ is analytic on $\CC\setminus(-\infty,0]$, $K_{\nu}(z)=K_{-\nu}(z)$ (see \cite[Equation (10.27.3)]{DLMF}) and let $\Gamma(\cdot)$ denote Euler's Gamma function. Then, the following asymptotic behaviour holds
for all $\nu > 0$
\begin{align}\label{eq:Kasymp}
   K_{\nu}(z)\sim_{z\to 0^+}
    \dfrac{\Gamma(\nu)}{2}\left(\dfrac{z}{2}\right)^{-\nu}  \quad\text{and}\quad
   K_{\nu}(z)\sim_{z\to\infty}\left(\dfrac{\pi}{2z}\right)^{1/2}e^{-z},
\end{align}
by \cite[Equation (10.30.2) and (10.30.3)]{DLMF} and \cite[Equation (10.25.3)]{DLMF}, respectively. 

\begin{proof}[Proof of Theorem~\ref{t-distinguished}]
	Let $\wt\la\in\sigma(J_N)\cap (0,1/2)$.
We claim that
	\begin{equation}\label{eq:regularity1}
		K_{\wt\la+1/2}(r)\f_{\wt\la}^j(\te)\notin H^{1/2}(\Omega_k,\CC^2),\qquad k\in\{1,2,\dots,N\},
	\end{equation}
	where, with a slight abuse of notation, $\f_{\wt\la}^j$ is understood to be its restriction to the interval $\II_k$.
	The space $H^{1/2}(\Omega_k,\CC^2)$ is embedded into $L^4(\Omega_k,\CC^2)$ by~\cite[Corollary 4.53]{DD07}. However, the asymptotics of $K_{\wt\la+1/2}(r)$ as $r\rightarrow0^+$ for $\nu = \wt\la + 1/2$ imply that $K_{\wt\la + 1/2}\f_{\wt\la}^j\notin L^4(\Omega_k,\CC^2)$ and hence $K_{\wt\la + 1/2}\f_{\wt\la}^j\notin H^{1/2}(\Omega_k,\CC^2)$, proving the claim.
	
	Furthermore, we claim 
	\begin{equation}\label{eq:regularity2}
	K_{\wt\la-1/2}(r)\f_{\wt\la}^j(\te)
	\in H^s(\Omega_k,\CC^2),\qquad k\in\{1,2,\dots,N\},~\text{for all}\,s<1/2+\wt\la.
	\end{equation}
	The Sobolev space $H^1_{\frac{2}{2-s}}(\Omega_k,\CC^2) = \{u\colon u,|\nabla u|\in L^{\frac{2}{2-s}}(\Omega_k,\CC^2)\}$
	is embedded into $H^s(\Omega_k,\CC^2)$ though (see e.g. \cite[Equation (1.4.4.5)]{G}), so it remains to verify that $
	K_{\wt\la-1/2}(r)\f_{\wt\la}^j
	\in H^1_{\frac{2}{2-s}}(\Omega_k,\CC^2)$.
	The asymptotics of $K_\nu$ in~\eqref{eq:Kasymp} yields that 
	$K_{\wt\la-1/2}\f_{\wt\la}^j
	\in L^{\frac{2}{2-s}}(\Omega_k,\CC^2)$.
	Combining the representation
	\[
		\big|\nabla (K_{\wt\la-1/2}\f_{\wt\la}^j)\big|^2(r,\te) = 
		| K_{\wt\la-1/2}'(r)|^2|\f_{\wt\la}^j(\te)|^2 + \frac{K^2_{\wt\la-1/2}(r)}{r^2}|(\f_{\wt\la}^j)'(\te)|^2,
	\]
	with the identity (cf. \cite[Equation (10.29.2)]{DLMF})
	\[	
		K'_{\wt\la-1/2}(r) = -K_{\wt\la-3/2}(r)-\frac{(\wt\la-1/2)K_{\wt\la-1/2}(r)}{r},
	\]
	and the asymptotics in~\eqref{eq:Kasymp}, we see that
	$|\nabla (K_{\wt\la-1/2}\f_{\wt\la}^j)|
	\in L^{\frac{2}{2-s}}(\Omega_k,\CC^2)$. The claim in equation~\eqref{eq:regularity2} follows.
	
	The decomposition from Theorem~\ref{t-nbreakdown} then says that the theorem holds if and only if the functions in  $\dom\fD_N^{\cU}$ have no terms with the factor $K\ci{\wt\la+1/2}(r)$ present. Clearly, this can happen only if $\cU = \cI_n$.
\end{proof}

\begin{rem}
	If $0\in\sigma(J_N)$, the operator $\fD_N$ has no self-adjoint extension whose domain is contained in
	$\oplus_{j=1}^N H^{1/2}(\Omega_j,\CC^2)$. In this case, the construction in the proof of Theorem~\ref{t-distinguished} can be repeated to observe that, for $\varepsilon >0$, a self-adjoint extension exists with domain contained in $\oplus_{j=1}^N H^{1/2-\varepsilon}(\Omega_j,\CC^2)$. However, this inclusion does not allow for a unique distinguished self-adjoint extension to be fixed.
	\hfill$\diamondsuit$
\end{rem}

\subsection*{Acknowledgement}
VL acknowledges the support by the grant No.~21-07129S 
of the Czech Science Foundation (GA\v{C}R). The authors would like to thank the referees for their thoughtful consideration of the manuscript, which has resulted in several improvements.

\begin{appendix}
\section{Closedness of the operator $\bd_{\wt\la}$}\label{app}
\begin{lem}
	The closure of the restriction $\bd_{\wt\la}\upharpoonright C_0^{\infty}(\RR_+,\CC^2)$ is $\bd_{\wt\la}$ with domain given by \eqref{eq:1dDirac}. In particular, $\bd_{\wt\la}$  with domain given by \eqref{eq:1dDirac} is closed.
\end{lem}
\begin{proof}
	We split the analysis into consideration of the two cases $\wt\la \ne 1/2$ and $\wt\la = 1/2$.
	
	Let $\wt\la \ne 1/2$.
	For any $\psi\in C^\infty_0(\RR_+,\CC^2)$, observe that
	\[
	\begin{aligned}
		\|\bd_{\wt\la}\psi&\|^2_{L^2(\RR_+,\CC^2)}
		\!=\!
		\int_{\RR_+}\left|\psi_2' + \frac{\wt\la\psi_2}{r}\right|^2d r
		+
		\int_{\RR_+}\left|\psi_1'- \frac{\wt\la\psi_1}{r}\right|^2d r\\
		& 
		\! =\!
		\int_{\RR_+}\left(|\psi_2'|^2 +
		\frac{2\wt\la \Re(\psi_2'\overline{\psi_2})}{r} +\frac{\wt\la^2|\psi_2|^2}{r^2}\right)d r
			\!+\!
			\int_{\RR_+}\left(|\psi_1'|^2 -
			\frac{2\wt\la \Re(\psi_1'\overline{\psi_1})}{r} +\frac{\wt\la^2|\psi_1|^2}{r^2}\right)d r\\
		& \!=\!
		\int_{\RR_+}\left(|\psi_2'|^2 +
		\frac{\wt\la(|\psi_2|^2)'}{r} +\frac{\wt\la^2|\psi_2|^2}{r^2}\right)d r
		\!+\!
		\int_{\RR_+}\left(|\psi_1'|^2 -
		\frac{\wt\la(|\psi_1|^2)'}{r} +\frac{\wt\la^2|\psi_1|^2}{r^2}\right)d r\\
		&\!=\!
		\int_{\RR_+}\left(|\psi_2'|^2 
		 +\frac{(\wt\la^2+\wt\la)|\psi_2|^2}{r^2}\right)d r
		+
		\int_{\RR_+}\left(|\psi_1'|^2 	+\frac{(\wt\la^2-\wt\la)|\psi_1|^2}{r^2}\right)d r,\\
	\end{aligned}	
	\]
	where in the last step we performed integration by parts.
	Combining the above formula with the one-dimensional Hardy inequality 
	\[
		\int_{\RR_+}|\psi'|^2dr \ge \int_{\RR_+}\frac{|\psi|^2}{4r^2}dr
		,\qquad 
		\forall\,\psi \in H^1_0(\RR_+),
	\]
	we conclude that the graph-norm induced by the operator $\bd_{\wt\la}$ on the space $C^\infty_0(\RR_+,\CC^2)$ is equivalent to the $H^1$-norm.
	Indeed, we get the double sided estimate
	\[
		\min\{1,(1-2\wt\la)^2\}\|\psi\|^2_{H^1(\RR_+,\CC^2)}   \le
		\|\bd_{\wt\la}\psi\|^2_{L^2(\RR_+,\CC^2)} +\|\psi\|^2_{L^2(\RR_+,\CC^2)}
		\le \left(1+2\wt\la\right)^2
		\|\psi\|^2_{H^1(\RR_+,\CC^2)}   
	\]	
	for any $\psi\in C^\infty_0(\RR_+,\CC^2)$.		
	Hence, we have $\bd_{\wt\la} = \overline{\bd_{\wt\la}\upharpoonright C^\infty_0(\RR_+,\CC^2)}$ and thus the operator $\bd_{\wt\la}$ is closed in this case.
		
	Let $\wt\la = 1/2$. Using \cite[Proposition 40]{DR} (see also \cite[Proposition 3.1]{BDG}) one can check that the domain of $\bd_{1/2}$ can be alternatively characterized as
	\[
		\dom\bd_{1/2} = \left\{\psi\in L^2(\RR_+,\CC^2)\colon
		\psi_1'-\frac{\psi_1}{2r},
		\psi_2'+\frac{\psi_2}{2r}\in L^2(\RR_+)\right\}.  
	\]	 
	Closedness of $\bd_{1/2}$ then follows from~\cite[Equation (1.10) and Theorem 1.1\,(ii)]{CP}.
\end{proof}
\end{appendix}


\end{document}